\documentclass{amsart}

\usepackage{amsmath}
\usepackage{amsfonts}
\usepackage{amstext}
\usepackage{amsthm}
\usepackage{amssymb}
\usepackage{latexsym}
\usepackage{graphicx}
\usepackage{texdraw}
\usepackage{graphpap}

\usepackage[pagebackref,hypertexnames=false, colorlinks, citecolor=red, linkcolor=blue]{hyperref} 
\usepackage{amsrefs}

\newtheorem{theorem}{Theorem}

\newtheorem{corollary}[theorem]{Corollary}

\newtheorem{problem}[theorem]{Problem}
\newtheorem{proposition}[theorem]{Proposition}
\newtheorem{remark}[theorem]{Remark}

\newcommand{\RR}{{\mathbb R}}
\newcommand{\CC}{{\mathbb C}}
\newcommand{\DD}{{\mathbb D}}
\renewcommand{\SS}{{\mathbb S}}

\newcommand{\NN}{{\mathbb N}}

\newcommand{\TTT}{{\mathcal T}}
\newcommand{\GGG}{{\mathcal G}}
\newcommand{\III}{{\mathcal I}}

\newcommand{\MMM}{{\mathcal M}}
\newcommand{\PPP}{{\mathcal P}}
\newcommand{\SSS}{{\mathcal S}}
\newcommand{\DDD}{{\mathcal D}}
\newcommand{\HHH}{{\mathcal H}}
\newcommand{\BBB}{{\mathcal B}}
\newcommand{\EEE}{{\mathcal E}}

\newcommand{\capacity}{{\mbox{Cap}}}

\begin{document}

\title{The Dirichlet Space: A Survey}

\author[N. Arcozzi]{Nicola Arcozzi}
\address{N. Arcozzi, Dipartimento do Matematica\\ Universita di Bologna\\ 40127 Bologna, ITALY}
\email{arcozzi@dm.unibo.it}
\author[R. Rochberg]{Richard Rochberg}
\address{R. Rochberg, Department of Mathematics\\
Washington University\\
St. Louis, MO 63130, U.S.A}
\email{rr@math.wustl.edu}
\author[E. T. Sawyer]{Eric T. Sawyer}
\address{E. T. Sawyer, Department of Mathematics \& Statistics\\
McMaster University\\
Hamilton, Ontairo, L8S 4K1, CANADA}
\email{sawyer@mcmaster.ca}
\author[B. D. Wick]{Brett D. Wick}
\address{B. D. Wick, School of Mathematics, Georgian Institute of Technology, 686 Cherry\\
Street, Atlanta, GA USA 30332--0160}
\email{wick@math.gatech.edu}
\thanks{N.A.'s work partially supported by the COFIN project Analisi Armonica, funded
by the Italian Minister for Research}
\thanks{R.R.'s work supported by the National Science Foundation under Grant No. 0700238}
\thanks{E.S.'s work supported by the National Science and Engineering Council of Canada.}
\thanks{B.W.'s work supported by the National Science Foundation under Grant No. 1001098}

\begin{abstract}
In this paper we survey many results on the Dirichlet space of analytic functions.  Our focus is more on the classical Dirichlet space on the disc and \textit{not} the potential generalizations to other domains or several variables.  Additionally, we focus mainly on certain function theoretic properties of the Dirichlet space and omit covering the interesting connections between this space and operator theory.  The results discussed in this survey show what is known about the Dirichlet space and compares it with the related results for the Hardy space.
\end{abstract}

\maketitle

\tableofcontents

\section{Introduction} 

{\footnotesize
\noindent\bf Notation. \rm The unit disc will be denoted by $\DD=\{z\in\CC:\ |z|<1\}$ and the
unit circle by $\SS=\partial\DD$. If $\Omega$ is open in $\CC$, $H(\Omega)$ is the space of the 
functions which are holomorphic in $\Omega$.
A function $\varphi:\SS\to\CC$ is identified with a function defined on $[0,2\pi)$;
$\varphi(e^{i\theta})=\varphi(\theta)$. 

Given two (variable) quantities $A$ and $B$, we write $A\approx B$ if there are universal 
constants $C_1,C_2>0$ such that $C_1 A\le B\le C_2 A$. Similarly, we use the symbol $\lesssim$.
If $A_1,\dots,A_n$ are mathematical objects, the symbol $C(A_1,\dots,A_n)$ denotes a
constant which only depends on $A_1,\dots,A_n$.
}

\medskip

The Dirichlet space, together with the Hardy and the Bergman space, is one of the three classical spaces of holomorphic functions in the unit disc.  Its theory is old, but over the past thirty years much has been learned about it and about the operators acting on it. The aim of this article is to survey some aspects, old and and new, of the Dirichlet theory. 

We will concentrate on the ``classical'' Dirichlet space and we will not dwell into its interesting extensions and generalizations. The only exception, because it is instrumental to our discourse, will be some discrete function spaces on trees.

Our main focus will be a \it Carleson-type \rm program, which has been unfolding over the past thirty years. In particular, to obtain a knowledge of the Dirichlet space comparable to that of the Hardy space $H^2$: weighted imbedding theorems (``Carleson measures''); interpolating sequences; the Corona Theorem. We also consider other topics which are well understood in the Hardy case:  bilinear forms; applications of Nevanlinna-Pick theory; spaces which are necessary to develop the Hilbert space theory ($H^1$ and $BMO$, for instance, in the case of $H^2$). Let us further mention a topic which is specifically related to the Dirichlet theory, namely the rich relationship with potential theory and capacity.

This survey is much less than comprehensive. We will be mainly interested in the properties of the Dirichlet space \it per se\rm, and will talk about the rich operator theory that has been developed on it when this intersects our main path.  We are also biased, more or less voluntarily, towards the topics on which we have been working.  If the scope of the survey is narrow, we will try to give some details of the ideas and arguments, in the hope to provide a service to those who for the first time approach the subject.

Let us finally mention the excellent survey \cite{Ro} by Ross on the Dirichlet space, to which we direct the reader for the discussion on the local Dirichlet integral, Carleson's and Douglas' formulas, and the theory of invariant subspaces. Also, \cite{Ro} contains a discussion of zero sets and boundary behavior. We will only tangentially touch on these topics here.  The article \cite{Wu} surveys some results in the operator theory on the Dirichlet space. 

\section{The Dirichlet Space} 
\subsection{The Definition of the Dirichlet Space} The \it Dirichlet space \rm $\DDD$ is the Hilbert space of 
analytic functions $f$ in the unit disc $\DD=\{z\in\CC:\ |z|<1\}$ for which the semi-norm
\begin{equation}
\label{seminorm1}
\|f\|_{\DDD,*}^2=\int_\DD|f^\prime(z)|^2dA(z)
\end{equation}
is finite. Here, $dA(x+iy)=\frac{1}{\pi}dxdy$ is normalized area measure. An easy calculation with Fourier coefficients shows that, if $f(z)=\sum_{n=0}^\infty a_nz^n,$
\begin{equation}
 \label{seminorm2}
\|f\|_{\DDD,*}^2=\sum_{n=1}^\infty n|a_n|^2.
\end{equation}
The Dirichlet space sits then inside the analytic Hardy space $H^2$. In particular, Dirichlet functions have nontangential limits at $a.e.$ point on the boundary of $\DD$. Much more can be said though, both on the kind of approach region and on the size of the exceptional set, see the papers \cite{NRS}, \cite{Ro} and \cite{Tw}.

There are different ways to make the semi-norm into a norm. Here, we use as norm and inner 
product, respectively,
\begin{eqnarray}\label{norm}
\|f\|_{\DDD}^2&=&\|f\|_{\DDD,*}^2+\|f\|_{H^2(\SS)}^2,\crcr 
\langle f,g\rangle_{\DDD}&=&\langle f,g\rangle_{\DDD,*}+\langle f,g\rangle_{H^2(\SS)}\crcr
&=&\int_\DD f^\prime(z)\overline{g^\prime(z)}dA(z)+\frac{1}{2\pi}\int_0^{2\pi}f(e^{i\theta})\overline{g(e^{i\theta})}d\theta.
\end{eqnarray}
Another possibility is to let $\vert\vert\vert f \vert\vert\vert_{\DDD}^2=\|f\|_{\DDD,*}^2+|f(0)|^2$. Most analysis on $\DDD$ carries out in the same way, no matter the chosen norm. There is an important exception to this rule. The \it Complete Nevanlinna-Pick Property \rm is not invariant under change of norm since it is satisfied by $\|\cdot\|_\DDD$, but \textit{not} by $\vert\vert\vert\cdot\vert\vert\vert_\DDD$.

The Dirichlet semi-norm has two different, interesting geometric interpretations.
\begin{itemize}
\item[\hypertarget{Area1}{(Area)}] Since $Jf=|f^\prime|^2$ is the Jacobian determinant of $f$,  
\begin{equation}\label{area}
 \|f\|_{\DDD,*}^2=\int_\DD dA(f(z))=A(f(\DD))
\end{equation}
is the area of the image of $f$, counting multiplicities.  This invariance property, which depends on the \it values \rm of functions in  $\DDD$, implies that the Dirichlet class is invariant under biholomorphisms of the disc.
\item[\hypertarget{Hyp}{(Hyp)}] Let $ds^2=\frac{|dz|^2}{(1-|z|^2)^2}$ be the hyperbolic metric 
in the unit disc. The (normalized) hyperbolic area density is 
$d\lambda(z)=\frac{dA(z)}{(1-|z|^2)^2}$ and the intrinsic derivative of a holomorphic
$f:(\DD,ds^2)\to(\CC,|dz|^2)$ is $\delta f(z)=(1-|z|^2)|f^\prime(z)|$. Then,
\begin{equation}\label{hyperbolic}
\|f\|_{\DDD,*}^2=\int_\DD(\delta f)^2d\lambda
\end{equation}
is defined in purely hyperbolic terms.
\end{itemize}
Since any Blaschke product with $n$ factors is an $n$-to-$1$ covering of the unit disc, \hyperlink{Area1}{(Area)} implies that the Dirichlet space only contains finite Blaschke products. 
On the positive side, \hyperlink{Area1}{(Area)} allows one to define the Dirichlet space on any simply connected domain $\Omega\subsetneq\CC$,
$$
\|f\|_{\DDD(\Omega),*}^2:=\int_\Omega|f^\prime(z)|^2dA(z)=\|f\circ\varphi\|_{\DDD,*}^2,
$$
where $\varphi$ is any conformal map of the unit disc onto $\Omega$. In particular, this shows
that the Dirichlet semi-norm is invariant under the M\"obius group $\MMM(\DD)$. 

Infinite Blaschke products provide examples of bounded functions which are not in 
the Dirichlet space. On the other hand, conformal maps of the unit disc onto unbounded regions
having finite area provide examples of unbounded Dirichlet functions.

\smallskip

The group $\MMM(\DD)$ acts on $(\DD,ds^2)$ as the group of the sense preserving isometries. 
It follows from \hyperlink{Hyp}{(Hyp)} as well, then, that  the Dirichlet semi-norm is conformally invariant: 
$\|f\circ\varphi\|_{\DDD,*}=\|f\|_{\DDD,*}$ when $\varphi\in\MMM(\DD)$.  In fact, in \cite{AF} Arazy and Fischer showed that the Dirichlet semi-norm is the only M\"obius invariant, Hilbert semi-norm for functions  holomorphic in the unit disc. Also, the Dirichlet space is the only M\"obius invariant Hilbert space of holomorphic functions on the unit disc.
Sometimes it is preferable 
to use the \it pseudo-hyperbolic \rm metric instead,
$$
\rho(z,w):=\left|\frac{z-w}{1-\overline{w}z}\right|.
$$
The hyperbolic metric $d$ 
and the pseudo-hyperbolic metric are functionally related,
$$
d=\frac{1}{2}\log\frac{1+\rho}{1-\rho},\ \ 
\rho=\frac{e^d-e^{-d}}{e^d+e^{-d}}.
$$
The hyperbolic metric is the only Riemannian metric which coincides with the pseudo-hyperbolic 
metric in the infinitesimally small. The triangle property for the hyperbolic metric is 
equivalent to an enhanced triangle property for the pseudo-hyperbolic metric:
$$
\rho(z,w)\le\frac{\rho(z,t)+\rho(t,w)}{1+\rho(z,t)\rho(t,w)}.
$$ 
We conclude with a simple and entertaining consequence of \hyperlink{Hyp}{(Hyp)}. The isoperimetric inequality
\begin{equation}
\label{isoperimetric}
\textnormal{Area}(\Omega)\le\frac{1}{4\pi}[\textnormal{Length}(\partial\Omega)]^2
\end{equation}
is equivalent, by Riemann's Mapping Theorem and by the extension of \eqref{isoperimetric}
itself to areas with multiplicities, to the inequality
$$
\|f\|^2_{\DDD,*}=\int_\DD|f^\prime|^2dA\le
\left[\frac{1}{2\pi}\int_{\partial\DD}|f^\prime(e^{i\theta})|d\theta\right]^2
=\|f^\prime\|_{H^1}^2.
$$
Setting $f^\prime=g$ in the last inequality, then the isoperimetric inequality becomes the imbedding of the Hardy space $H^1$ into the Bergman space $A^2$ with optimal constant:
$$
\|g\|_{A^2}^2\le\|g\|_{H^1}^2,
$$
the constant functions being extremal.
\subsubsection{The Hardy space $H^2$.} The ``classical'' Hilbert spaces of holomorphic functions on the unit disc are the Dirichlet space just introduced, the Bergman space $A^2$,
$$
\|f\|_{A^2}^2=\int_\DD|f(z)|^2dA(z),
$$
and the Hardy space $H^2$,
$$
\|f\|_{H^2}^2=\sup_{0<r<1}\frac{1}{2\pi}\int_0^{2\pi}|f(re^{i\theta})|^2d\theta.
$$
The Hardy space is especially important because of its direct r\^ole in operator theory, as a prototype for the study of boundary problems for elliptic differential equations, for its analogy with important probabilistic objects (martingales), and for many other reasons. It has been studied in depth and its theory has become a model for the theory of other classical, and not so classical, function spaces. Many results surveyed in this article have been first proved, in a different version, for the Hardy space.

It is interesting to observe that both the Hardy and the Bergman space can be thought of as weighted Dirichlet spaces. We consider here the case of the Hardy space. If $f(0)=0$, then
$$
\|f\|_{H^2}^2=\int_\DD|f^\prime(z)|^2\log\frac{1}{|z|^{2}}dA(z)\approx\int_\DD|f^\prime(z)|^2(1-|z|^2)dA(z).
$$
This representation of the Hardy functions is more than a curiosity. Since $H^2$ is a reproducing kernel 
Hilbert space (RKHS) of functions, we are interested in having a norm which depends on the values of $f$ in the \it interior \rm of the unit disc. (Indeed, the usual norm is in terms of interior values as well, although through the mediation of $\sup$).

\subsection{The Definition in terms of Boundary Values and other Characterizations
of the Dirichlet Norm} Let $\SS=\partial\DD$ be the unit circle
and $\HHH^{1/2}(\SS)$ be the fractional Sobolev space containing the functions 
$\varphi\in L^2(\SS)$ having ``$1/2$'' derivative in $L^2(\SS)$. More precisely, if
$\varphi(\theta)=\sum_{n=1}^{+\infty}[a_n\cos(n\theta)+b_n\sin(n\theta)]$, then the $\HHH^{1/2}(\SS)$
semi-norm of $\varphi$ is
\begin{equation}\label{sobolev}
\|\varphi\|_{\HHH^{1/2}(\SS)}^2=\sum_{n=1}^{+\infty}n(|a_n|^2+|b_n|^2).
\end{equation}
By definition,
$$
\DDD\equiv\HHH^{1/2}(\SS)\cap H(\DD).
$$
This is a special instance the fact that, restricting Sobolev functions from the plane to smooth curves, 
``there is a loss of $1/2$ derivative''.
\subsubsection{The Definition of Rochberg and Wu.} In \cite{RW}, Rochberg and Wu gave a characterization of the Dirichlet norm in terms
of difference quotients of the function.
\begin{theorem}[Rochberg and Wu, \cite{RW}]
\label{RW}
Let $\sigma,\tau>-1$.
For an analytic function $f$ on the unit disc we have the semi-norm equivalence:
$$
\|f\|_{\DDD,*}^2\approx 
\int_\DD\int_\DD\frac{|f(z)-f(w)|^2}{|1-z\overline{w}|^{\sigma+\tau+4}}(1-|z|^2)^\sigma(1-|w|^2)^\tau dA(w)dA(z).
$$
\end{theorem}
For $\sigma=\tau=1/2$, the Theorem holds with equality instead of approximate equality; see \cite{AFP}. The result in \cite{RW} extends to weighted Dirichlet spaces and, with a different, essentially, discrete proof, to analytic Besov spaces \cite{BlPa}. The characterization in Theorem \ref{RW} is similar in spirit to the usual boundary characterization for functions in $\HHH^{1/2}(\SS)$:
$$
\|\varphi\|_{\HHH^{1/2}(\SS)}^2\approx\int_0^{2\pi}\int_0^{2\pi}\frac{|\varphi(\zeta)-\varphi(\xi)|^2}{|\zeta-\xi|^{2}}d\zeta d\xi.
$$
\subsubsection{The characterization of B\"oe.} In \cite{Bo2}, B\"oe obtained an interesting characterization of the norm in analytic Besov spaces in terms of the mean oscillation of the function's modulus with respect to harmonic measure. We give B\"oe's result in the Dirichlet case.
\begin{theorem}[B\"oe, \cite{Bo2}]
\label{bjarte} For $z\in\DD$, let 
$$
d\mu_z(e^{i\theta})=\frac{1}{2\pi}\frac{1-|z|^2}{|e^{i\theta}-z|^2}d\theta
$$
be harmonic measure on $\SS$ with respect to $z$. Then,
$$
\|f\|_{\DDD,*}^2\approx
\int_\DD\left(\int_0^{2\pi}|f(e^{i\theta})|d\mu_{z}(e^{i\theta})-|f(z)|\right)^2\frac{dA(z)}{(1-|z|^2)^2}.
$$
\end{theorem}
\subsection{The Reproducing Kernel}
 The space $\DDD$ has bounded point evaluation 
$\eta_z:f\mapsto f(z)$ at each point $z\in\DD$. Equivalently, 
it has a reproducing kernel. In fact, it is easily checked that
$$
f(z)=\langle f,K_z\rangle_{\DDD},\ \mbox{with}\ K_z(w)=\frac{1}{\overline{z}w}\log\frac{1}{1-\overline{z}w}.
$$
(For the norm $\||\cdot|\|_\DDD$ introduced earlier, the reproducing kernel is
$$
\tilde{K}_z(w)=1+\log\frac{1}{1-\overline{z}w}
$$
which is comfortable in estimates for the integral operator having $\tilde{K}_z(w)$ as kernel).

It is a general fact that $\|\eta_z\|_{\DDD^*}=\|K_z\|_\DDD$ and an easy calculation gives
$$
\|K_z\|_\DDD^2\approx1+\log\frac{1}{1-|z|}\approx 1+d(z,0).
$$
More generally, we have that functions in the Dirichlet space are H\"older continuous of
order $1/2$ with respect to the hyperbolic distance:
\begin{equation}\label{Holder}
|f(z)-f(w)|\le C\|f\|_{\DDD,*}d(z,w)^{1/2}.
\end{equation}
The reproducing kernel $K_z(w)=K(z,w)$ satisfies some estimates which are important in
applications and reveal its geometric nature:
\begin{itemize}
\item[(a)] $\Re K(z,w)\approx |K(z,w)|$ (here and below, $\Re(x+iy)=x$ is the real part of $x+iy$); 
\item[(b)] Let $z\wedge w$ be the point which is closest to the origin (in either the hyperbolic
or Euclidean metric) on the hyperbolic geodesic joining $z$ and $w$. Then,  
$$
\Re K(z,w)\approx d(0,z\wedge w)+1;
$$
\item[(c)] $\frac{d\ }{dw}K(z,w)=\frac{\overline{z}}{1-\overline{z}w}$, and we have:
\begin{itemize}
\item[(c1)] $\Re\frac{1}{1-\overline{z}w}\ge0$ for all $z,w$ in $\DD$; 
\item[(c2)] $\Re \frac{1}{1-\overline{z}w}\approx(1-|z|^2)^{-1}$ for $w\in S(z)$, where
$$
S(z)=\{w\in\DD:\ |1-\overline{z}w|\le 1-|z|^2\}
$$
is the Carleson box with centre $z$.
\end{itemize}
\end{itemize}
\section{Carleson measures} 
\subsection{Definition and the Capacitary Characterization} A positive Borel measure measure $\mu$ on $\overline{\DD}$ is 
called a \it Carleson measure \rm for the Dirichlet space if for some finite $C>0$
\begin{equation}\label{carleson}
\int_{\overline{\DD}}|f|^2d\mu\le C\|f\|_\DDD^2\quad\forall f\in\DDD.
\end{equation}
The smallest $C$ in \eqref{carleson} is the \it Carleson measure norm \rm of $\mu$ and 
it will be denoted by
$[\mu]=[\mu]_{CM(\DDD)}$. The space of the Carleson measures for $\DDD$ is denoted
by $CM(\DDD)$. Carleson measures supported on the boundary could be thought of 
as substitutes for point evaluation (which is not well defined at boundary points). 
By definition, in fact, the function $f$ exists, in a 
quantitative way, on sets which support a strictly positive Carleson measure. It is then
to be expected that there is a relationship between Carleson measures and boundary values of 
Dirichlet functions.  This is further explained below.

Carleson measures proved to be a central concept in the theory of the Dirichlet space 
in many other ways. Let us mention:
\begin{itemize}
 \item Multipliers; 
 \item Interpolating Sequences; 
 \item Bilinear Forms; 
 \item Boundary Values.
\end{itemize} 
Since Carleson measures play such an important role, it is important to have efficient ways to characterize them. The first such characterization was given by Stegenga \cite{Ste} in terms of capacity.

We first introduce the Riesz-Bessel kernel of order $1/2$ on $\SS$,
\begin{equation}\label{besselkernel}
k_{\SS,1/2}(\theta,\eta)=|\theta-\eta|^{-1/2},
\end{equation}
where the difference $\theta-\eta\in[-\pi,\pi)$ is taken modulo $2\pi$. The kernel extends to a 
convolution operator, which we still call $k_{\SS,1/2}$, acting on Borel measures supported on
$\SS$,
$$
k_{\SS,1/2}\nu(\theta)=\int_\SS k_{\SS,1/2}(\theta-\eta)d\nu(\eta).
$$
Let $E\subseteq\SS$ be a closed set. The $(\SS,1/2)$-Bessel capacity of $E$ is
\begin{equation}\label{capacity}
\capacity_{\SS,1/2}(E):=\inf\left\{\|h\|_{L^2(\SS)}^2:
\ h\ge0\ \mbox{and}
\ k_{1/2,\SS}h\ge1\ \mbox{on}\ E\right\}.
\end{equation}
It is a well known fact \cite{Stein} that 
$\|k_{\SS,1/2}h\|_{\HHH^{1/2}(\SS)}\approx \|h\|_{L^2(\SS)}$, i.e., that 
$h\mapsto k_{\SS,1/2}h$ is an approximate isometry of $L^2(\SS)$ into $\HHH^{1/2}(\SS)$. Hence,
$$
\capacity_{\SS,1/2}(E)\approx\inf\left\{\|\varphi\|_{\HHH^{1/2}(\SS)}^2:
(k_{\SS,1/2})^{-1}\varphi\ge0\ \mbox{and}\ \varphi\ge1\ \mbox{on}\ E\right\}.
$$
\begin{theorem}[Stegenga, \cite{Ste}]
\label{stegenga}Let $\mu\ge0$ be a positive Borel measure on $\DD$. Then $\mu$
is Carleson for $\DDD$ if and only if there is a positive constant $C(\mu)$ such that, 
for any choice of finitely many disjoint, closed arcs $I_1,\dots,I_n\subseteq\SS$, we have that
\begin{equation}\label{stegengacondition}
\mu\left(\cup_{i=1}^n S(I_i)\right)\le C(\mu)\capacity_{\SS,1/2}\left(\cup_{i=1}^n I_i\right).
\end{equation}
Moreover, $C(\mu)\approx[\mu]_{CM(\DDD)}$. 
\end{theorem}
It is expected that capacity plays a r\^{o}le in the theory of the Dirichlet space. In fact,
as we have seen, the Dirichlet space is intimately related to at least two Sobolev spaces 
($\HHH^{1/2}(\SS)$ and $\HHH^{1}(\CC)$, which is defined below), and capacity plays in Sobolev
theory the r\^{o}le played by Lebesgue measure in the theory of Hardy spaces. In Dirichlet space
theory, this fact has been recognized for a long time see, for instance, \cite{Beu}; actually, before Sobolev theory reached maturity.

It is a useful exercise comparing Stegenga's capacitary condition and Carleson's condition for the Carleson measures for the Hardy space. In \cite{Car2} Carleson proved that for a positive Borel measure $\mu$ on $\DD$,
$$
\int_\DD|f|^2d\mu\le C(\mu)\|f\|_{H^2}^2\ \iff\ \mu(S(I))\le C^\prime(\mu)|I|,
$$
for all closed sub-arcs $I$ of the unit circle. Moreover, the best constants in the two 
inequalities are comparable. In some sense, Carleson's characterization says that $\mu$
satisfies the imbedding $H^2\hookrightarrow L^2(\mu)$ if and only if it behaves 
(no worse than) the arclength measure on $\SS$, the measure underlying the Hardy theory.
We could also ``explain'' Carleson's condition in terms of the reproducing kernel for the
Hardy space,
$$
K_z^{H^2}(w)=\frac{1}{1-\overline{z}w},\ \|K^{H^2}_z\|_{H^2}^2\approx(1-|z|)^{-1}.
$$
Let $I_z$ be the arc having center in $z/|z|$ and arclength $2\pi(1-|z|)$. 
Carleson's condition can then be rephrased as
$$
\mu(S(I_z))\le C(\mu)\|K^{H^2}_z\|_{H^2}^2.
$$
Similar conditions hold for the (weighted) Bergman spaces. One might expect that a 
necessary and sufficient condition for a measure to belong to belong to $CM(\DDD)$ might be
\begin{equation}\label{simplecondition}
\mu(S(I_z))\le C(\mu)\|K_z\|_{\DDD}\approx\frac{1}{\log\frac{2}{1-|z|}}\approx
\capacity_{\SS,1/2}(I_z).
\end{equation}
The  ``simple condition'' \eqref{simplecondition} is necessary, but not sufficient. 
Essentially, this follows from the fact that the simple condition does not ``add up''.  If
$I_j$, $j=1,\dots,2^n$, are adjacent arcs having the same length, and $I$ is their union,
then
$$
\sum_j\capacity_{\SS,1/2}(I_j)\approx \frac{2^n}{(\log2) n+\log\frac{4\pi}{|I|}}
>\log\frac{1}{\frac{4\pi}{|I|}}\approx\capacity_{\SS,1/2}(I).
$$
Stegenga's Theorem has counterparts in the theory of Sobolev spaces where the problem is that of
finding necessary and sufficient conditions on a measure $\mu$ so that a \it trace inequality \rm holds. For instance, consider the case of the Sobolev space $\HHH^1(\RR^n)$, containing those functions $h:\RR^n\to\CC$ with finite norm
$$
\|h\|_{\HHH^1(\RR^n)}^2=\|h\|_{L^2(\RR^n)}^2+\|\nabla h\|_{L^2(\RR^n)}^2,
$$
the gradient being the distributional one. The positive Borel measure $\mu$ on $\RR^n$ satisfies
a trace inequality for $\HHH^1(\RR^n)$ if the imbedding inequality
\begin{equation}\label{traceinequality}
\int_{\RR^n}|h|^2d\mu\le C(\mu)\|h\|_{\HHH^1(\RR^n)}^2
\end{equation}
holds. It turns out that \eqref{traceinequality} is equivalent to the condition that
\begin{equation}\label{tracecondition}
\mu(E)\le C(\mu)\capacity_{\HHH^1(\RR^n)}(E)
\end{equation}
holds for all compact subsets $E\subseteq\RR^n$. Here, $\capacity_{\HHH^1(\RR^n)}(E)$ is 
the capacity naturally associated with the space $\HHH^1(\RR^n)$.

There is an extensive literature on trace inequalities, which is closely related 
to the study of Carleson measures for the Dirichlet space and its extensions. We will not discuss it further, but instead direct the interested reader to \cite{Adams}, \cite{KaVe}, \cite{KS1} and \cite{Maz}, for
a first approach to the subject from different perspectives.

Complex analysts may be more familiar with the logarithmic capacity, than with Bessel capacities.
It is a classical fact that, for subsets $E$ of the unit circle (or of the real line)
\begin{equation}\label{comparisonofcapacities}
\capacity_{\SS,1/2}(E)\approx\log\gamma(E)^{-1},
\end{equation}
where $\gamma(E)$ is the \it logarithmic capacity \rm (the \it transfinite diameter\rm) of the
set $E$.
\subsection{Characterizations by Testing Conditions} The capacitary condition has to be checked
over all finite unions of arcs. It is natural to wonder whether there is a ``single box'' 
condition characterizing the Carleson measures. In fact, there is a string of such 
conditions, which we are now going to discuss. The following statement rephrases the 
characterization given in \cite{ARS1}. Let $k(z,w)=\Re K(z,w)$.
\begin{theorem}[Arcozzi, Rochberg and Sawyer, \cite{ARS1}]
\label{newtestingcondition}
Let $\mu$ be a positive Borel measure on $\overline{\DD}$. Then $\mu$ is a Carleson measure 
for $\DDD$ if and only if $\mu$ is finite and
\begin{equation}
\label{eqnewtesting}
\int_{\overline{S(\zeta)}}\int_{\overline{S(\zeta)}}k(z,w)d\mu(w)d\mu(z)\le C(\mu)\mu(\overline{S(\zeta)})
\end{equation}
for all $\zeta$ in $\DD$.

Moreover, if $C_{\mbox{best}}(\mu)$ is the best constant in \eqref{eqnewtesting}, then 
$$
[\mu]_{CM(\mu)}\approx C_{\mbox{best}}(\mu)+\mu(\overline{\DD}).
$$
\end{theorem}
The actual result in \cite{ARS1} is stated differently. There, it is shown that 
$\mu\in CM(\DDD)$ if and only if $\mu$ is finite and
\begin{equation}\label{eqtesting}
\int_{S(\zeta)}\mu(\overline{S(z)}\cap \overline{S(\zeta)})^2\frac{dA(z)}{(1-|z|^2)^2}\le C(\mu)\mu(\overline{S(\zeta)}),
\end{equation}
with $[\mu]_{CM(\mu)}\approx C_{\mbox{best}}(\mu)+\mu(\overline{\DD})$. The equivalence between 
these two conditions will be discussed below, when we will have at our disposal the 
simple language of trees.

\proof[Proof discussion] The basic tools are a duality argument and two weight 
inequalities for positive kernels. It is instructive to enter in some detail the 
duality arguments. The definition of Carleson measure says that the imbedding
$$
Id:\ \DDD\hookrightarrow L^2(\mu) 
$$
is bounded. Passing to the adjoint $\Theta=Id^*$, this is equivalent to the boundedness of
$$
\Theta:\ \DDD\hookleftarrow L^2(\mu).
$$
The adjoint makes ``unstructured'' $L^2(\mu)$ functions into holomorphic functions,
so we expect it to be more manageable. Using the reproducing kernel property,
we see that, for $g\in L^2(\mu)$
\begin{eqnarray}\label{firstdualitystep}
\Theta g(\zeta)&=&\langle \Theta g,K_\zeta\rangle_\DDD\crcr
&=&\langle g,K_\zeta\rangle_{L^2(\mu)}\crcr
&=&\int_{\overline{\DD}}g(z)K_z(\zeta)d\mu(z),
\end{eqnarray}
because $\overline{K_\zeta(z)}=K_z(\zeta)$. We now insert \eqref{firstdualitystep} in 
the boundedness property of $\Theta g$:
\begin{eqnarray*}
 C(\mu)\int_{\overline{\DD}}|g|^2d\mu&\ge&\|\Theta g\|_\DDD^2\crcr
&=&\left\langle\int_{\overline{\DD}}g(z)K_z(\cdot)d\mu(z),
\int_{\overline{\DD}}g(w)K_w(\cdot)d\mu(w)\right\rangle_\DDD\crcr
&=&\int_{\overline{\DD}}g(z)\int_{\overline{\DD}}\overline{g(w)}d\mu(w)\langle K_z,K_w\rangle_\DDD d\mu(z)\crcr
&=&\int_{\overline{\DD}}g(z)\int_{\overline{\DD}}\overline{g(w)}d\mu(w)K_z(w)d\mu(z).
\end{eqnarray*}
Overall, we have that the measure $\mu$ is Carleson for $\DDD$ if and only if the 
weighted quadratic inequality
\begin{equation}\label{seconddualitystep}
 \int_{\overline{\DD}}g(z)\int_{\overline{\DD}}\overline{g(w)}d\mu(w)K_z(w)d\mu(z)\le
C(\mu)\int_{\overline{\DD}}|g|^2d\mu
\end{equation}
holds. Recalling that $k(z,w)=\Re K_z(w)$, it is clear that \eqref{seconddualitystep} implies
\begin{equation}\label{thirddualitystep}
\int_{\overline{\DD}}g(z)\int_{\overline{\DD}}g(w)d\mu(w)k(z,w)d\mu(z)\le
C(\mu)\int_{\overline{\DD}}|g|^2d\mu,
\end{equation}
for real valued $g$ and that, vice-versa, \eqref{thirddualitystep} for real valued $g$ implies
\eqref{seconddualitystep}, with a twice larger constant: 
$\|\Theta (g_1+i g_2)\|_{\DDD}^2\le2(\|\Theta g_1\|_{\DDD}^2+\|\Theta g_2\|_{\DDD}^2)$. 
The same reasoning says that $\mu$ is Carleson if and only if \eqref{thirddualitystep}
holds for positive $g$'s since the problem is reduced to a weighted inequality for a real (positive, 
in fact), symmetric kernel $k$. Condition \eqref{eqnewtesting} is obtained by testing 
\eqref{thirddualitystep} over functions of the form $g=\chi_{\overline{S(\zeta)}}$. The 
finiteness of $\mu$ follows by testing the imbedding $\DDD\hookrightarrow L^2(\mu)$ on the
function $f\equiv1$.

The hard part is proving the sufficiency of \eqref{eqnewtesting}: see \cite{ARS2}, \cite{ARS3}, \cite{ARS1},  
 \cite{KaVe}, \cite{Tch} for different approaches to the problem. See also the very recent \cite{VoWi} for an approach covering the full range of the weighted Dirichlet spaces  in the unit ball of $\CC^n$, between unweighted Dirichlet and Hardy.
\endproof

The reasoning above works the same way with all reproducing kernels (provided the integrals 
involved make sense, of course). In particular, the problem of finding the Carleson measures
for a RKHS reduces, in general, to a weighted quadratic inequality like \eqref{thirddualitystep},
with positive $g$'s.

\subsubsection{A Family of Necessary and Sufficient Testing Conditions.} Condition 
\eqref{newtestingcondition} is the endpoint of a family of such conditions, and the 
quadratic inequality \eqref{thirddualitystep} is the endpoint of a corresponding family of 
quadratic inequalities equivalent to the membership of $\mu$ to the Carleson class.

The kernels $K$ and $k=\Re K$ define positive operators on $\DDD$, hence, by general Hilbert space theory, the boundedness in the inequality
$$
\|\Theta g\|_\DDD^2\le C(\mu)\|g\|_{L^2(\mu)}^2
$$
is equivalent to the boundedness of the operator 
$S:g\mapsto Sf=\int_{\overline{\DD}}k(\cdot,w)g(w)d\mu(w)$ on $L^2(\mu)$, i.e., to
\begin{equation}\label{fourthdualitystep}
\int_{\overline{\DD}}\left(\int_{\overline{\DD}}k(z,w)g(w)d\mu(w)\right)^2d\mu(z)
\le C(\mu)\int_{\overline{\DD}}g^2d\mu,
\end{equation}
with the same constant $C(\mu)$. Testing \eqref{fourthdualitystep} on 
$g=\chi_{\overline{S(\zeta)}}$ and restricting, we have the new testing condition
\begin{equation}\label{newnewtestingcondition}
\int_{\overline{S(\zeta)}}\left(\int_{\overline{S(\zeta)}}k(z,w)d\mu(w)\right)^2d\mu(z)
\le C(\mu)\mu(\overline{S(\zeta)}).
\end{equation}
Observe that, by Jensen's inequality, \eqref{newnewtestingcondition} is \it a priori \rm 
stronger than 
\eqref{newtestingcondition}, although, by the preceding considerations, it is equivalent to it.
Assuming the viewpoint that \eqref{fourthdualitystep} represents the $L^2(\mu)\to L^2(\mu)$
inequality for the ``singular integral operator'' having kernel $k$, and using 
sophisticated machinery used to solve the Painlev\'e problem,
Tchoundja \cite{Tch} pushed this kind of analysis much further. 
Using also results in \cite{ARS2}, he was able to prove the following.
\begin{theorem}[Tchoundja, \cite{Tch}]
\label{tchoundja}
Each of the following conditions on a finite measure $\mu$ 
is equivalent to the fact that $\mu\in CM(\DDD)$:
\begin{itemize}
\item For some $p\in(1,\infty)$ the following inequality holds,
\begin{equation}\label{streamone}
\int_{\overline{\DD}}\left(\int_{\overline{\DD}}k(z,w)g(w)d\mu(w)\right)^pd\mu(z)
\le C_p(\mu)\int_{\overline{\DD}}g^pd\mu.
\end{equation}
\item Inequality \eqref{streamone} holds for all $p\in(1,\infty)$.
\item For some $p\in[1,\infty)$ the following testing condition holds,
\begin{equation}\label{streamtwo}
\int_{\overline{S(\zeta)}}\left(\int_{\overline{S(\zeta)}}k(z,w)d\mu(w)\right)^pd\mu(z)
\le C_p(\mu)\mu(\overline{S(\zeta)}).
\end{equation}
\item The testing condition \eqref{streamtwo} holds for all $p\in[1,\infty)$.
\end{itemize}
\end{theorem}
Actually, Tchoundja deals with different spaces of holomorphic functions, but his results
extend to the Dirichlet case. As mentioned earlier, the 
$p=1$ endpoint of Theorem \ref{tchoundja} is in \cite{ARS2}.
\subsubsection{Another Family of Testing Conditions} It was proved in \cite{ARS1} 
that a measure $\mu$ on $\overline{\DD}$ is Carleson for $\DDD$ if and only if
\eqref{eqtesting} holds. In \cite{KS2}, Kerman and Sawyer had found another, seemingly weaker,
necessary and sufficient condition. In order to compare the two conditions, we restate 
\eqref{eqtesting} differently. Let $I(z)=\partial S(z)\cap\partial\DD$. 
For $\theta\in I(z)$ and $s\in[0,1-|z|]$, let $S(\theta,s)=S((1-s)e^{i\theta})$. Condition 
\eqref{eqtesting} is easily seen to be equivalent to have, for all $z\in\DD$,
\begin{equation}\label{neweqtesting}
\int_{I(z)}
\int_0^{1-|z|}\left(\frac{\mu(S(z)\cap\mu(S(\theta,s)))}{s^{1/2}}\right)^2\frac{ds}{s}d\theta
\le C(\mu)\mu(S(z)).
\end{equation}
Kerman and Sawyer proved that $\mu$ is a Carleson measure for $\DDD$ if and only if
for all $z\in\DD$,
\begin{equation}\label{kseqtesting}
\int_{I(z)}
\sup_{s\in(0,1-|z|]}\left(\frac{\mu(S(z)\cap\mu(S(\theta,s)))}{s^{1/2}}\right)^2d\theta
\le C(\mu)\mu(S(z)).
\end{equation}
Now, the quantity inside the integral on the left hand side of \eqref{kseqtesting} is smaller than the 
corresponding quantity in \eqref{neweqtesting}. Due to the presence of the measure
$ds/s$ and the fact that the quantity $\mu(S(\theta,s))$ changes regularly with $\theta$ fixed
and $s$ variable the domination of the left hand side \eqref{kseqtesting} by that of \eqref{neweqtesting} 
comes from the imbedding 
$\ell^2\subseteq\ell^\infty$. The fact that, ``on average'', the inclusion can be reversed is
at first surprising. In fact, it is a consequence of the Muckenhoupt-Wheeden inequality \cite{MW} 
(or an extension of it), that the quantities on the left hand side of \eqref{neweqtesting} and 
\eqref{kseqtesting} are equivalent.
\begin{theorem}[\cite{ARS1} \cite{KS2}]
\label{ARSKS}
A measure $\mu$ on $\overline{\DD}$ is Carleson for the Dirichlet space $\DDD$ if and only if it
is finite and for some $p\in[1,\infty]$ (or, which is the same, {\rm for all} $p\in[1,\infty]$)
and all $z\in\DD$:
\begin{equation}\label{pcondition}
\int_{I(z)}
\left[\int_0^{1-|z|}\left(\frac{\mu(S(z)\cap\mu(S(\theta,s)))}{s^{1/2}}\right)^p\right]^{2/p}
\frac{ds}{s}d\theta\le C(\mu)\mu(S(z)).
\end{equation}
\end{theorem}
The inequality of Muckenhoupt and Wheeden was independently rediscovered by T. Wolff \cite{HW}, 
with a completely new proof. Years later, trying to understand why the conditions in 
\cite{ARS1} and \cite{KS2} where equivalent, although seemingly different, in \cite{AR} the 
authors, unaware of the results in \cite{HW} and \cite{MW}, found another (direct) proof
of the inequality.

\section{The Tree Model} 
\subsection{The Bergman Tree}
The unit disc $\DD$ can be discretized into Whitney boxes. The set of 
such boxes has a natural tree structure. In this section, we want to explain how analysis on
the holomorphic Dirichlet space is related to analysis on similar spaces on the tree, and not
only on a metaphoric level.

For integer $n\ge0$ and $1\le k\le 2^n$, consider the regions
$$
Q(n,k)=\left\{z=re^{i\theta}\in\DD:\ 2^{-n}\le1-|z|<2^{-n-1},\ 
\frac{k}{2^n}\le\frac{\theta}{2\pi}<\frac{k+1}{2^n}\right\}.
$$
Let $\TTT$ be the set of the indices $\alpha=(n,k)$. Sometimes we will identify the index $\alpha$ with the region $Q(\alpha)$.
The regions indexed by $\TTT$ form a partition of the unit disc
$\DD$ in regions, whose Euclidean diameter, Euclidean in-radius, and Euclidean distance to the 
boundary are comparable to each other, with constants independent of the considered region. 
An easy exercise in hyperbolic geometry shows that the regions
$\alpha\in\TTT$ have approximatively the same hyperbolic diameter and hyperbolic in-radius.
We give the set $\TTT$ two geometric-combinatorial structures: a \it tree structure, \rm in 
which there is an edge between $\alpha$ and $\beta$ when the corresponding regions share an 
arc of a circle; a \it graph structure, \rm in which there is an edge between $\alpha$ and 
$\beta$  if the closures of the corresponding regions have some point of $\DD$ in common. When 
referring to the graph structure, we write $\GGG$ instead of $\TTT$. 

In the tree $\TTT$, we choose a distinguished point $o=\alpha(0,1)$, the \it root \rm of $\TTT$.
The distance $d_\TTT(\alpha,\beta)$
 between two points $\alpha,\beta$ in $\TTT$ is the minimum number of edges of $\TTT$
one has to travel going from the vertex $\alpha$ to the vertex $\beta$. Clearly, there 
is a unique path from $\alpha$ to $\beta$ having minimal length: it is the \it geodesic \rm 
$[\alpha,\beta]$ between $\alpha$ and $\beta$, which we consider as a set of points. 
The choice of the root gives $\TTT$ a partial order structure: $\alpha\le\beta$ if 
$\alpha\in[o,\beta]$. The \it parent \rm of $\alpha\in\TTT\setminus\{o\}$ is the point 
$\alpha^{-1}$ on $[o,\alpha]$ such that $d(\alpha,\alpha^{-1})=1$. Each point $\alpha$ is the parent 
of two points in $\TTT$ (its \it children\rm), labeled when necessary as $\alpha^\pm$.
The natural geometry on $\TTT$ is a simplified version of the hyperbolic geometry of the disc.

We might define a distance $d_\GGG$ on the graph $\GGG$ using edges of $\GGG$ instead of 
edges of $\TTT$. The distance $d_\GGG$ is realized by geodesics, although we do not have 
uniqueness anymore. However, we have ``almost uniqueness'' in this case, two geodesic between $\alpha$ and $\beta$ maintain a reciprocal distance which is bounded by a positive constant $C$, independent
of $\alpha$ and $\beta$. The following facts are rather easy to prove:
\begin{enumerate}
\item[(1)] $d_\GGG(\alpha,\beta)\le d_\TTT(\alpha,\beta)$;
\item[(2)] If $z\in\alpha$ and $w\in\beta$, then 
$d_\GGG(\alpha,\beta)+1\approx d(z,w)+1$: the graph metric is roughly the hyperbolic metric
at unit scale;
\item[(3)] There are sequences $\{\alpha_n\},\ \{\beta_n\}$ such that
$d_\TTT(\alpha_n,\beta_n)/d_\GGG(\alpha_n,\beta_n)\to\infty$ as $n\to\infty$.  This says there are points which are close in the graph, but far away in the tree.
\end{enumerate}
While the graph geometry is a good approximation of the hyperbolic geometry at a fixed scale,
the same can not be said about the tree geometry. Nevertheless, the tree geometry is much
more elementary, and it is that we are going to use. It is a bit surprising that, in spite 
of the distortion of the hyperbolic metric pointed out in (3), the tree geometry is so useful.

Let us introduce the analogs of cones and Carleson boxes on the tree: 
$\PPP(\alpha)=[o,\alpha]\subset\TTT$ is the \it predecessor set \rm of $\alpha\in\TTT$ 
(when you try to picture it, you get a sort of cone) and 
$\SSS(\alpha)=\{\beta\in\TTT:\ \alpha\in\PPP(\alpha)\}$, its dual object, is the \it
successor set \rm of $\alpha$ (a sort of Carleson box).

Given $\alpha,\beta\in\TTT$, we denote their \it confluent \rm by $\alpha\wedge\beta$.  This is  the point
on the geodesic between $\alpha$ and $\beta$ which is closest to the root $o$. That is,
$$
\PPP(\alpha\wedge\beta)=\PPP(\alpha)\cap\PPP(\beta).
$$
In terms of $\DD$ geometry, the confluent corresponds to the highest point of the smallest Carleson box
containing two points; if $z,w\in\DD$ are the points, the point which plays the r\^ole
of $\alpha\wedge\beta$ is roughly the point having argument halfway between that of $z$ and
that of $w$, and having Euclidean distance $|1-z\overline{w}|$ from the boundary.

\subsection{Detour: The Boundary of the Tree and its Relation with the Disc's Boundary.}
The distortion of the metric induced by the tree structure has an interesting 
effect on the boundary. One can define a boundary $\partial\TTT$ of the tree $\TTT$. While
the boundary of $\DD$ (which we might think of as a boundary for the graph $\GGG$) is connected,
the boundary $\partial\TTT$ is totally disconnected; it is in fact homeomorphic to a Cantor set.
Notions of boundaries for graphs, and trees in particular, are an old and useful topic 
in probability and potential theory. We mention \cite{Saw} as a nice introduction to this topic.

We will see promptly that the boundary $\partial\TTT$ is compact with respect to a natural metric and that, as such, it carries positive Borel measures. Furthermore, if $\mu$ is positive Borel measure without atoms with support on $\partial\DD$, then it can be identified with a positive Borel measure without atoms on $\partial\TTT$.

This is the main reason we are interested in trees and a tree's boundary. Some theorems are 
easier to prove on the tree's boundary, some estimates become more transparent and some
objects are easier to picture. Often, it is possible to split a problem in two parts: a ``soft''
part, to deal with in the disc geometry, and a ``hard'' combinatorial part, which one can
formulate and solve in the easier tree geometry.
Many of these results and objects can then be transplanted in
the context of the Dirichlet space.

\smallskip

As a set, the boundary $\partial\TTT$ contains as elements the half-infinite geodesics on $\TTT$, 
having $o$ as endpoint. For convenience, we think of $\zeta\in\partial\TTT$ as of a point 
and we denote by $\PPP(\zeta)=[o,\zeta)\subset\TTT$ the geodesic labeled by $\zeta$. We 
introduce on $\partial\TTT$ a metric which mimics the Euclidean metric on the circle:
$$
\delta_\TTT(\zeta,\xi)=2^{-d_\TTT(\zeta\wedge\xi)},
$$
where $\zeta\wedge\xi$ is defined as in the ``finite'' case $\alpha,\beta\in\TTT$: 
$\PPP(\zeta\wedge\xi)=\PPP(\zeta)\cap\PPP(\xi)$. It is easily verified that, modulo a 
multiplicative constant, $\delta_\TTT$ is the weighted length of the doubly infinite geodesic 
$\gamma(\zeta,\xi)$ which joins $\zeta$ and $\xi$, where the weight assigns to each edge
$[\alpha,\alpha^{-1}]$ the number $2^{-d_\TTT(\alpha)}$. The metric can be extended 
to $\overline{\TTT}=\TTT\cup\partial\TTT$ by similarly measuring a geodesics' lengths for 
all geodesics. This way, we obtain a compact metric space $(\overline{\TTT},\delta_\TTT)$, where
$\TTT$ is a discrete subset of $\overline{\TTT}$, having $\partial\TTT$ as metric boundary. 
\it The subset $\partial\TTT$, 
as we said before, turns out to be a totally disconnected, perfect set. \rm

The relationship between $\partial\TTT$ and $\partial\DD$ is more than metaphoric. Given a point 
$\zeta\in\partial\TTT$, let $\PPP(\zeta)=\{\zeta_n:\ n\in\NN\}$ be an enumeration of the points
$\zeta_n\in\TTT$ of the corresponding geodesic, ordered in such a way that $d(\zeta_n,o)=n$. 
Each $\alpha$ in $\TTT$ can be identified with a dyadic sub-arc of $\partial\DD$. If $Q(\alpha)$ 
is the Whitney box labeled by $\alpha=(n,k)$, let
$$
S(\alpha)=\left\{z=re^{i\theta}\in\DD:\ 2^{-n}\le1-|z|,\ 
\frac{k}{2^n}\le\frac{\theta}{2\pi}<\frac{k+1}{2^n}\right\}
$$
be the corresponding \it Carleson box\rm. Consider the arc 
$I(\alpha)\partial S(\alpha)\cap\partial\DD$ and
define the map
$\Lambda:\partial\TTT\to\partial\DD$,
\begin{equation}\label{lambda}
\Lambda(\zeta)=\bigcap_{n\in\NN}I(\zeta_n).
\end{equation}
It is easily verified that $\Lambda$ is a Lipschitz continuous map of $\partial\TTT$ onto 
$\partial\DD$, which fails to be injective at a countable set (the set of the dyadic rationals 
$\times 2\pi$). More important is the (elementary, but not obvious) fact that $\Lambda$ maps
Borel measurable sets in $\partial\TTT$ to Borel measurable sets in $\partial\DD$. This allows
us to move Borel measures back and forth from $\partial\TTT$ to $\partial\DD$.

Given a positive Borel measure $\omega$ on $\partial \TTT$, let $(\Lambda_*\omega)(E)=\omega(\Lambda^{-1}(E))$ be the usual push-forward measure. Given a positive Borel measure $\mu$ on $\SS$, define its pull-back  $\Lambda^*\mu$ to be the positive Borel measure on $\partial \TTT$\rm
\begin{equation}\label{pulback}
(\Lambda^*\mu)(F)=\int_\SS\frac{\sharp(\Lambda^{-1}(\theta)\cap A)}{\sharp(\Lambda^{-1}(\theta))}d\mu(e^{i\theta}).
\end{equation}
\begin{proposition}
\label{lambdatwo}
\begin{enumerate}
\item[]
 \item[(i)] The integrand in \eqref{pulback} is measurable, hence the integral is well-defined; 
 \item[(ii)] $\Lambda_*(\Lambda^*(\mu))=\mu$;
 \item[(iii)] For any closed subset $A$ of $\SS$, $\Lambda_*\omega(A)=\omega(\Lambda^{-1}(A))$, by definition;
 \item[(iv)] For any closed subset $B$ of $\partial\TTT$, $\Lambda^*\mu(B)\approx\mu(\Lambda(B))$;
 \item[(v)] In \textnormal{(iv)}, we have equality if the measure $\mu$ has no atoms.
\end{enumerate}
\end{proposition}
See \cite{ARSW3} for more general versions of the proposition.
\subsection{A Version of the Dirichlet Space on the Tree}  Consider the Hardy-type operator
$\III$ acting on functions $\varphi:\TTT\to\RR$,
$$
\III\varphi(\alpha)=\sum_{\beta\in \PPP(\alpha)}\varphi(\beta).
$$
The \it Dirichlet space \rm $\DDD_\TTT$ on $\TTT$ is the space of the functions 
$\Phi=\III\varphi$, $\varphi\in\ell^2(\TTT)$, with norm 
$\|\Phi\|_{\DDD_\TTT}=\|\varphi\|_{\ell^2}$. Actually, we will always talk about the space 
$\ell^2$ \it and \rm the operator $\III$, rather than about the space $\DDD_\TTT$, which is 
however the \it trait d'union \rm between the discrete and the continuous theory.

What we are thinking of, in fact, is discretizing a Dirichlet function $f\in\DDD$ in such a way that
\begin{enumerate}
\item[(1)] $\varphi(\alpha)\sim (1-|z(\alpha)|)|f^\prime(z(\alpha))|$, where $z(\alpha)$
is a distinguished point in the region $\alpha$ (or in its closure);
\item[(2)] $\III\varphi(\alpha)=f(\alpha)$.
\end{enumerate}

Let us mention a simple example from \cite{ARS6}, saying that $\ell^2$ is ``larger'' than $\DDD$.
\begin{proposition}\label{restriction} Consider a subset $\{z(\alpha):\ \alpha\in\TTT\}$
of $\DD$, where $z(\alpha)\in\alpha$, and let $f\in\DDD$. Then, there is a function $\varphi$
in $\ell^2(\TTT)$ such that $\III\varphi(\alpha)=f(z(\alpha))$ for all $\alpha\in\TTT$ and
$\|\varphi\|_{\ell^2}\lesssim\|f\|_\DDD$.
\end{proposition}
\proof Assume without loss of generality that $f(0)=0$ and let 
$\varphi(\alpha):=f(z(\alpha))-f(z(\alpha^{-1}))$. By telescoping, 
$\varphi(\alpha)=f(z(\alpha))$. To prove the estimate,
\begin{eqnarray*}
 \|h\|_{\ell^2(\TTT)}^p&=&\sum_{\alpha}\left|f(z({\alpha}))-f(z({\alpha^{-1}}))\right|^2\crcr
&\lesssim&\sum_{\alpha}\left|(1-|z(\alpha)|)f^\prime(w(\alpha))\right|^2\crcr
&\ &\mbox{for\ some\ $w(\alpha)$\ in\ the\ closure\ of\ $\alpha$,}\crcr
&\approx&\sum_{\alpha}(1-|z(\alpha)|)^2\left|\frac{1}{(1-|z_\alpha|)^2}
\int_{\zeta\in\DD:\ |\zeta-w(\alpha)|\le(1-|z(\alpha)|)/10}f^\prime(\zeta)dA(\zeta)
\right|^2\crcr
&\ &\mbox{by\ the\ (local)\ Mean Value Property,}\crcr
&\lesssim&
\sum_{\alpha}
\int_{\zeta:\ |\zeta-w(\alpha)|\le(1-|z(\alpha)|)/10}\left|f^\prime(\zeta)\right|^2dA(\zeta)
\crcr
&\ &\mbox{by\ Jensen's\ inequality,}\crcr
&\approx&\|f\|_{\DDD}^2,
\end{eqnarray*}
since the discs $\left\{\zeta:\ |\zeta-w(\alpha)|\le\frac{1-|z(\alpha)|}{10}\right\}$ clearly have bounded 
overlap.
\endproof
\subsection{Carleson Measures on the Tree and on the Disc} Let $\mu$ be a positive 
measure on the closed unit disc. Identify it with a positive measure on $\TTT$ by letting
$$
\mu(\alpha)=\int_{Q(\alpha)} d\mu(z).
$$
\subsubsection{Carleson Measures.} We say that $\mu$ is a \it Carleson measure \rm for $\DDD_\TTT$
if the operator $\III:\ell^2(\TTT)\to\ell^2(\TTT,\mu)$ is bounded. 
We write $\mu\in CM(\DDD_\TTT)$.
\begin{theorem}\label{discretocontinuo}
We have that $CM(\DDD)=CM(\DDD_\TTT)$ with comparable norms.
\end{theorem}
\proof[Proof discussion] We can use the restriction argument of Proposition \ref{restriction}
to show that $CM(\TTT)\subseteq CM(\DDD)$. Suppose for simplicity that $\mu(\partial\DD)=0$
(dealing with this more general case requires further discussion of the tree's boundary) and that
$\mu\in CM(\TTT)$:
\begin{eqnarray*}
\int_{\DD}|f|^2d\mu&=&\sum_{\alpha}\int_\alpha|f|^2d\mu\le
\sum_{\alpha}\mu(\alpha)|f(z(\alpha))|^2\crcr
&\ &\mbox{for\ some\ $z(\alpha)$\ on\ the\ boundary\ of\ $\alpha$}\crcr
&=&\sum_\alpha\III\varphi(\alpha)\mu(\alpha)\crcr
&\ &\mbox{with\ $\varphi$\ as\ in\ Proposition\ \ref{restriction}}\crcr
&\le&\|\varphi\|_{\ell^2(\TTT)}^2,
\end{eqnarray*}
which proves the inclusion. 

In the other direction, we use the duality argument used in the proof
of Theorem \ref{newtestingcondition}. The fact that $\mu$ is Carleson for $\DDD$ is equivalent to
the boundedness of $\Theta$, the adjoint of the imbedding, and this is equivalent to the 
inequality
\begin{eqnarray}\label{anothercountry}
C(\mu)\int_\DD|g|^2d\mu&\ge&\|\Theta g\|_\DDD^2=
\int_{\DD}\left|(\Theta g)^\prime(z)\right|^2dA(z)\crcr
&\ &\mbox{this\ time\ we\ use\ a\ different\ way\ to\ compute\ the\ norm,}\crcr
&\ge&\int_\DD\left|\int_{\DD}\frac{d\ }{dz}K(z,w)g(w)d\mu(w)\right|^2dA(z)\crcr
&=&\int_\DD\left|\int_{\DD}\frac{\overline{w}}{1-\overline{w}z}g(w)d\mu(w)\right|^2dA(z).
\end{eqnarray}
Testing \eqref{anothercountry} over all functions $g(w)=h(w)/\overline{w}$ with $h\ge0$ and 
using the geometric properties of the kernel's derivative, we see that
\begin{equation}\label{bordeaux}
C(\mu)\int_\DD|g|^2d\mu\ge 
\int_\DD\left|\int_{S(z)}\overline{w}g(w)d\mu(w)\right|^2dA(z).
\end{equation}
We can further restrict to the case where $h$ is constant on Whitney boxes 
($h=\sum_{\alpha\in\TTT}\psi(\alpha)\chi_\alpha$) and, further restricting the integral, 
we see that \eqref{bordeaux}
reduces to
\begin{equation}\label{nicola}
C(\mu)\|\psi\|_{\ell^2(\mu)}^2\ge\|\III^*(\psi d\mu)\|_{\ell^2}^2.
\end{equation}
A duality argument similar (in the converse direction) to the previous one, this time in tree-based function spaces, shows that the last assertion is equivalent to having $\III:\ell^2(\TTT)\to\ell^2(\TTT,\mu)$ bounded, 
i.e., $\mu\in CM(\TTT)$.
\endproof

The proof could be carried out completely in the dual side. Actually, this is almost obliged
in several extensions of the theorem 
(to higher dimensions \cite{ARS6}, to ``sub-diagonal'' couple of indices \cite{A},
etcetera). 
A critical analysis of the proof and some further considerations about the boundary 
of the tree show that Carleson measures satisfy a stronger property.
\begin{corollary}[Arcozzi, Rochberg, and Sawyer, \cite{ARS3}]
\label{radialvariation}
Let 
$$V(f)(Re^{i\theta})=\int_0^R|f^\prime(re^{i\theta})|dr
$$
be the radial variation of $f\in\DDD$ (i.e., the length of the image of the radius $[0,Re^{i\theta}]$ under $f$). Then, $\mu\in CM(\DDD)$ if and only if the stronger inequality
$$
\int_{\overline{\DD}}V(f)^2d\mu\le C(\mu)\|f\|_\DDD^2
$$
holds.
\end{corollary}
Indeed, this remark is meaningful when $\mu$ is supported on $\partial\DD$.

\subsubsection{Testing Conditions in the Tree Language}
In the proof discussion following Theorem \ref{discretocontinuo}, we ended by showing 
that a necessary and sufficient condition for a measure $\mu$ on $\overline{\DD}$ to be in 
$CM(\DD)$ is \eqref{nicola}. Making duality explicit, one computes 
$$
\III^*(\psi d\mu)(\alpha)=\int_{\overline{\SSS(\alpha)}}gd\mu.
$$
Using as testing functions $g=\chi_{\overline{\SSS(\alpha_0)}},\ \alpha_0\in\TTT$ and throwing
away some terms on the right hand side, we obtain the discrete testing condition:
\begin{eqnarray}\label{treedual}
C(\mu)\mu(\overline{\SSS(\alpha_0)})&\ge&\sum_{\alpha\in\SSS(\alpha_0)}
[\mu(\overline{\SSS(\alpha)})]^2.
\end{eqnarray}
We will denote by $[\mu]$ the best constant in \eqref{treedual}.
\begin{theorem}[Arcozzi, Rochberg, and Sawyer, \cite{ARS1}]\label{anoldone}
A measure $\mu$ on $\overline{\DD}$ belongs to $CM(\DDD)$ if, and only if, it is finite
and it satisfies \eqref{treedual}.
\end{theorem}
Given Theorem \ref{discretocontinuo}, Theorem \ref{anoldone} really becomes a characterization
of the weighted inequalities for the operator $\III$ (and/or its adjoint). There is a vast 
literature on weighted inequalities for operators having positive kernels, 
and virtually all of the proofs translate in the present context. Theorem \ref{anoldone} 
was proved in \cite{ARS1} by means of a \it good-$\lambda$ \rm argument. A different proof
could be deduced by the methods in \cite{KaVe}, where a deep equivalence is established between
weighted inequalities and a class of integral (nonlinear) equations. In \cite{ARS3} a very short 
proof is given in terms of a maximal inequality.

The fact that the (discrete) testing condition \eqref{treedual} characterizes Carleson measures raises two natural questions:
\begin{itemize}
\item Is there a \it direct \rm proof that the testing condition \eqref{treedual} is 
equivalent to Stegenga's capacitary condition?
\item Is there an ``explanation'' of how a condition which is expressed in terms of the 
\it tree structure \rm is sufficient to characterize properties whose natural environment is
the \it graph structure \rm of the unit disc? 
\end{itemize}
\subsubsection{Capacities on the Tree.} Let $E$ be a closed subset of $\partial\TTT$. We define a logarithmic-type and a Bessel-type capacity for $E$. As in the continuous case, they turn out to be equivalent. 

The operator $\III$ can be extended in the obvious way on the boundary of the tree, $\III\varphi(\zeta)=\sum_{\beta\in\PPP(\zeta)}\varphi(\beta)$ for $\zeta$ in $\partial\TTT$.  Then, 
\begin{equation}\label{logtree}
\capacity_{T}(E)=\inf\left\{\|\varphi\|_{\ell^2(\TTT)}^2:\ \III\varphi(\zeta)\ge1\ \mbox{on}\ E\right\}
\end{equation}
will be the \it tree capacity \rm of $E$, which roughly corresponds to logarithmic capacity.

Define the kernel $k_{\partial\TTT}:\partial\TTT\times\partial\TTT\to[0,+\infty]$,
$$
k_{\partial\TTT}(\zeta,\xi)=2^{d_\TTT(\zeta\wedge\xi)/2},
$$
which mimics the Bessel kernel $k_{\SS,1/2}$. 
The energy of a measure $\omega$ on $\partial\TTT$ associated with the kernel is
$$
\EEE_{\partial\TTT}(\omega)=\int_{\partial\TTT}\left(k_{\partial\TTT}\omega(\zeta)\right)^2dm_{\partial\TTT}(\zeta),
$$
where $m_{\partial\TTT}=\Lambda^*m$ is the pullback of the linear measure on $\SS$. More concretely, $m_{\partial\TTT}\partial\SSS(\alpha)=2^{-d_\TTT(\alpha)}$.
We define another capacity
$$
\capacity_{\partial\TTT}(E)=\sup\left\{\frac{\omega(E)^2}{\EEE_{\partial\TTT}(\omega)}:\ \mbox{supp}(\omega)\subseteq E\right\},
$$
the supremum being taken over positive, Borel measures on $\partial\TTT$.

As in the continuous case (with a simpler proof) one has that the two capacities are equivalent,
$$
\capacity_{\TTT}(E)\approx\capacity_{\partial\TTT}(E).
$$
It is not obvious that both are equivalent to the logarithmic capacity.
\begin{theorem}[Benjamini and Peres, \cite{BePe}]
\label{BePe}
$$
\capacity_{\TTT}(E)\approx\capacity_{\partial\TTT}(E)\approx\capacity_{\SS,1/2}(\Lambda(E)).
$$
\end{theorem}
See \cite{ARSW3} for an extension of this result to Bessel-type capacities on Ahlfors-regular 
metric spaces.

\begin{proof}[Proof discussion] Let $\omega$ be a positive Borel measure on $\partial\TTT$ and $\mu$
be a positive Borel measure on $\SS$. It suffices to show that the energy 
of $\omega$ with respect to the kernel $k_{\partial\TTT}$ is comparable with the energy of $\Lambda_*\omega$ 
with respect to $k_{\SS,1/2}$ and that the energy of $\mu$ is comparable with energy of $\Lambda^*\mu$, with respect to to the same kernels, obviously taken in reverse order. We can also assume the measures to be atomless,
since atoms, both in $\SS$ and $\partial\TTT$, have infinite energy. Proposition \ref{lambdatwo}
implies that the measure $\Lambda^*\mu$ is well defined and helps with the energy estimates, which 
are rather elementary. 
\end{proof}

Theorem \ref{BePe} has direct applications to the theory of the Dirichlet space.
\begin{itemize}
\item As explained in \cite{ARSW3}, there is a direct relationship between tree capacity 
$\capacity_\TTT$ and Carleson measures for the Dirichlet space. Let $[\mu]$ be the best value $C(\mu)$ in \eqref{treedual}.
Namely, for a closed subset $E$ of 
$\partial\TTT$,
\begin{equation}\label{stasera}
\capacity_\TTT(E)=\sup_{\mu:\ \mbox{supp}(\mu)\subseteq E}\frac{\mu(E)}{[\mu]}.
\end{equation}
\item As a consequence, we have that sets having null capacity are exactly sets which do not 
support positive Carleson measures. Together with Corollary \eqref{radialvariation} and the 
theorem of Benjamini and Peres, this fact 
implies an old theorem by Beurling.
\begin{theorem}[Beurling, \cite{Beu}]\label{beurling}
$$
\capacity_{\SS,1/2}(\{\zeta\in\SS:\ V(f)(e^{i\theta})=+\infty\})=0.
$$
\end{theorem}
In particular, Dirichlet functions have boundary values at all points on $\SS$, but for a subset
having null capacity. This result, the basis for the study of boundary 
behavior of Dirichlet functions, explains the differences and similarities between Hardy and Dirichlet 
theories. It makes it clear that capacity is for $\DDD$  what arclength measure is in $H^2$. 
On the other hand, there are Hardy functions (even bounded analytic functions) having infinite 
radial variation at almost all points on $\SS$.  Radial variation is for the most part a 
peculiarly Dirichlet topic.
\item Another application is in \cite{ARSW1}, where boundedness of certain bilinear forms on $\DDD$
is discussed (and which also contains a different proof of Theorem \cite{BePe}, of which we were not aware at the 
moment of writing the article). Central to the proof of the main result is the 
holomorphic approximation of the discrete potentials which are extremal for the tree capacity
of certain sets. See Section \ref{intrinseco} for a discussion of this and related topics.
\end{itemize}
\subsubsection{Capacitary Conditions and Testing Conditions.} The capacitary condition of Stegenga and 
the discrete testing condition \eqref{treedual} (plus boundedness of $\mu$) are equivalent, since 
both characterize $CM(\DDD)$. It is easy to see that the capacitary condition is \it a priori \rm 
stronger than the testing condition. A \it direct \rm proof that the testing condition implies
the capacitary condition is in \cite{ARS4}. The main tool in the proof is the characterization 
\eqref{stasera} of the tree capacity.

\section{The Complete Nevanlinna-Pick property}

In 1916 Georg Pick published the solution to the following interpolation
problem. 

\begin{problem}
Given domain points $\left\{  z_{i}\right\}  _{i=1}^{n}\subset\mathbb{D}$ and
target points $\left\{  w_{i}\right\}  _{i=1}^{n}\subset\mathbb{D}$ what is a
necessary and sufficient condition for there to an $f\in H^{\infty},$
$\left\Vert f\right\Vert _{\infty}\leq1$ which solves the interpolation
problem $f\left(  z_{i}\right)  =w_{i}$ $i=1,...,n?$
\end{problem}

A few years later Rolf Nevanlinna independently found an alternative solution. The problem is now sometimes called Pick's problem and sometimes goes with both names; Pick-Nevanlinna (chronological) and Nevanlinna-Pick (alphabetical). The result has been extraordinarily influential.

One modern extension of Pick's question is the following:

\begin{problem}
[Pick Interpolation Question]\label{NP}Suppose $H$ is a Hilbert space of
holomorphic functions on $\mathbb{D}$. Given $\left\{  z_{i}\right\}
_{i=1}^{n},\left\{  w_{i}\right\}  _{i=1}^{n}\subset\mathbb{D}$ is there a
function $m$ in $\mathcal{M}_{H}$, the multiplier algebra, with $\left\Vert
m\right\Vert _{\mathcal{M}_{H}}\leq1,$ which performs the interpolation
$m(z_{i})=w_{i};$ $i=1,2,...,n$
\end{problem}

There is a necessary condition for the interpolation problem to have a solution which holds for any RKHS. We develop that now.  Suppose we are given the data for the interpolation question. 

\begin{theorem}
Let $V$ be the span of the kernel functions $\left\{  k_{i}\right\}  _{i=1}^{n}$.  Define the map $T$ by
\[
T\left(\sum a_{i}k_{i}\right)=\sum a_{i}\bar{w}_{i}k_{i}.
\]
A necessary condition for the Pick Interpolation Question to have a positive
answer is that $\left\Vert T\right\Vert \leq1.$ Equivalently a necessary
condition is that the associated matrix%
\begin{equation}
\operatorname*{Mx}(T)=\left(  \left(  1-w_{j}\bar{w}_{i}\right)  k_{j}\left(
z_{i}\right)  \right)  _{i,j=1}^{n}\label{pick}%
\end{equation}
be positive semi-definite; $\operatorname*{Mx}(T)\geq0.$
\end{theorem}
\begin{proof}
Suppose there is such a multiplier $m$ and let $M$ be the operator of
multiplication by $m$ acting on $H.$ We have $\left\Vert M\right\Vert
=\left\Vert m\right\Vert _{\mathcal{M}(H)}\leq1.$ Hence the adjoint operator,
$M^{\ast}$ satisfies $\left\Vert M^{\ast}\right\Vert \leq1.$ We know that
given $\zeta\in\mathbb{D}$, $M^{\ast}k_{\zeta}=\overline{m(\zeta)}k_{\zeta}$. 
 Thus $V$ is an invariant subspace for $M^{\ast}$ and the
restriction of $M^{\ast}$ to $V$ is the operator $T$ of the theorem. Also the
restriction of $M^{\ast}$ to $V$ has, \textit{a fortiori, norm at most one.
That gives the first statement.}

The fact that the norm of $T$ is at most one means that for \textit{scalars
}$\left\{  a_{i}\right\}  _{i=1}^{n}$ we have
\[
\left\Vert \sum a_{i}\bar{w}_{i}k_{i}\right\Vert ^{2}\leq\left\Vert \sum
a_{i}k_{i}\right\Vert ^{2}.
\]
We compute the norms explicitly recalling that $\left\langle k_{i}%
.k_{j}\right\rangle =k_{i}\left(  z_{j}\right)  $ and rearrange the terms and
find that
\[
\sum_{i,j}\left(  1-w_{j}\bar{w}_{i}\right)  k_{j}\left(  z_{i}\right)
a_{j}\bar{a}_{i}\geq0.
\]
The scalars $\left\{  a_{i}\right\}  _{i=1}^{n}$ were arbitrary and thus this
is the condition that $\operatorname*{Mx}(T)\geq0.$
\end{proof}

The matrix $\operatorname*{Mx}(T)$ is called the \textit{Pick matrix} of the
problem. For the Hardy space it takes the form%
\[
\operatorname*{Mx}(T)=\left(  \frac{1-w_{i}\bar{w}_{j}}{1-z_{i}\bar{z}_{j}%
}\right)  _{i,j=1}^{n}.%
\]

\begin{theorem}
[Pick]For the Hardy space, the necessary condition for the interpolation problem to have a solution, \eqref{pick}, is also sufficient.
\end{theorem}
See \cite{AgMc2} for a proof.
\begin{remark}
The analog of Pick's theorem fails for the Bergman space; \eqref{pick} is not sufficient. 
\end{remark}
It is now understood that there are classes of RKHSs for which condition \eqref{pick} is sufficient for the interpolation problem to have a solution.  Such spaces are said to have the Pick property. In fact there is a subclass, those with the Complete Nevanlinna Pick Property, denoted CNPP, for which \eqref{pick} is a sufficient condition for the interpolation problem to have a solution, and for a matricial analog of the interpolation problem to have a solution.

It is a consequence of the general theory of spaces with CNPP that the
kernel functions never vanish; $\forall z,w\in X,$ $k_{z}(w)\neq0.$   For spaces of the type we are considering there is a surprisingly simple
characterization of spaces with the CNPP. Suppose $H$ is a Hilbert space of
holomorphic functions on the disk in which the monomials $\left\{
z^{n}\right\}  _{n=0}^{\infty}$ are a complete orthogonal set. The argument we
used to identify the reproducing kernel for the Dirichlet space can be used
again and we find that for $\zeta\in\mathbb{D}$ we have
\begin{align*}
k_{\zeta}^{H}(z)  & =\sum_{n=0}^{\infty}\frac{\bar{\zeta}^{n}z^{n}}{\left\Vert
z^{n}\right\Vert _{H}^{2}}\\
& =\sum_{n=0}^{\infty}a_{n}\bar{\zeta}^{n}z^{n}.%
\end{align*}
We know that $a_{0}=\left\Vert 1\right\Vert _{H}^{-2}>0$ hence in a
neighborhood of the origin the function $\sum_{n=0}^{\infty}a_{n}t^{n}$ has a
reciprocal given by a power series. Define $\left\{  c_{n}\right\}  $ by
\begin{equation}
\frac{1}{\sum_{n=0}^{\infty}a_{n}t^{n}}=\sum_{n=0}^{\infty}c_{n}%
t^{n}.\label{recip}%
\end{equation}
Having $a_{0}>0$ insures $c_{0}>0.$

\begin{theorem}
The space $H$ has the CNPP if and only if
\[
c_{n}\leq0\text{ \ \ \ }\forall n>0.
\]
\end{theorem}

Using this we immediately see that the Hardy space has the CNPP and the Bergman
space does not. 
\begin{theorem}\label{CNPP}
The Dirichlet space $\DDD$ with the norm $\|\cdot\|_\DDD$,
\[
\left\Vert \sum_{n=0}^{\infty}b_{n}z^{n}\right\Vert _{\mathcal{D}}^{2}=\sum_{n=0}^{\infty}\left(  n+1\right)  \left\vert
b_{n}\right\vert ^{2},%
\] 
has the complete Nevanlinna-Pick property.
\end{theorem}

On the other hand one needs only compute a few of the $c_{n}$ to find out that:

\begin{remark}
The space $\mathcal{D}$ with the norm
\[
\left\vert\left\vert \left\vert \sum_{n=0}^{\infty}b_{n}z^{n} \right\vert\right\vert \right\vert  _{\mathcal{D}}%
^{2}=\left\vert b_{0}\right\vert ^{2}+\sum_{n=1}^{\infty}n\left\vert
b_{n}\right\vert ^{2}%
\]
does not have the CNPP.
\end{remark}

If a RKHS has the CNPP then a number of other subtle and interesting consequences follow. In particular, this applies for the Dirichlet space.  We refer the reader to the foundational article \cite{AgMc1} and to the beautiful monograph \cite{AgMc2} for a comprehensive introduction to spaces with the CNPP.

\section{\texorpdfstring{The Multiplier Space and other Spaces intrinsic to $\DDD$ Theory}{The Multiplier Space and other Spaces Intrinsic to Dirichlet Space Theory}}
\label{intrinseco}
\subsection{Multipliers}
Suppose $H$ is a RKHS of holomorphic functions in the disk. We say that a
function $m$ is a \textit{multiplier} (of $H$ or for $H$)\textit{\ }if
multiplication by $m$ maps $H$ boundedly to itself; that is there is a
$C=C(m)$ so that for all $h\in H$%
\[
\left\Vert mh\right\Vert _{H}\leq C\left\Vert h\right\Vert _{H}.
\]
Let $\mathcal{M}_{H}$ be the space of all multipliers of $H$ and for
$m\in\mathcal{M}_{H}$ let $\left\Vert m\right\Vert _{\mathcal{M}_{H}}$ be the
operator norm of the multiplication operator. With this norm $\mathcal{M}_{H}
$ is a commutative Banach algebra.

It is sometimes easy and sometimes difficult to get a complete description of
the multipliers of a given space $H.$ If the constant functions are in $H$
(they are, in the case of the Hardy and of the Dirichlet space), then $\mathcal{M}%
_{H}\subset H.$ In fact for $\left\Vert 1\right\Vert _{H}=1$ and hence the
inclusion is contractive: $\left\Vert m\right\Vert _{H}=\left\Vert
m\cdot1\right\Vert _{\mathcal{M}_{H}}\leq\left\Vert m\right\Vert
_{\mathcal{M}_{H}}\left\Vert 1\right\Vert _{H}=\left\Vert m\right\Vert
_{\mathcal{M}_{H}}. $

Also, for each of $\mathcal{D},$ $H^{2},$ and $A^{2}$ (the Bergman space) the multiplier algebra
is contractively contained in $H^{\infty},$
\[
\left\Vert m\right\Vert _{H^{\infty}}\leq\left\Vert m\right\Vert
_{\mathcal{M}_{H}}.
\]
One way to see this is by looking at the action of the adjoint of the
multiplication operator on reproducing kernels. Let $H$ be one of
$\mathcal{D},$ $H^{2},$ and $A^{2}$; let $m\in\mathcal{M}_{H}$ and let $M$ be
the operator of multiplication by $m$ acting on $H.$ Let $M^{\ast}$ be the
adjoint of the operator $M$. We select $\zeta,z\in\mathbb{D}$ and compute%
\begin{align*}
M^{\ast}k_{H,\zeta}(z)  & =\left\langle M^{\ast}k_{H,\zeta},k_{H,z}%
\right\rangle \\
& =\left\langle k_{H,\zeta},Mk_{H,z}\right\rangle \\
& =\left\langle k_{H,\zeta},mk_{H,z}\right\rangle \\
& =\overline{\left\langle mk_{H,z},k_{H,\zeta}\right\rangle }\\
& =\overline{m(\zeta)k_{H,z}\left(  \zeta\right)  }\\
& =\overline{m(\zeta)}k_{H,\zeta}(z).
\end{align*}
\it Thus $k_{H,\zeta}$ is an eigenvector of
$M^{\ast}$, the adjoint of the multiplication operator, 
with eigenvalue $\overline{m(\zeta)}.$ \rm Hence $\left\vert
m(\zeta)\right\vert \leq\left\Vert M^{\ast}\right\Vert =\left\Vert
M\right\Vert .$ Taking the supremum over $\zeta\in\mathbb{D}$ gives the
desired estimate.  For the Hardy space that is the full story; $\mathcal{M}_{H^{2}}=H^{\infty
}.$ In the Dirichlet case, things are a bit more complicated. 
\begin{proposition}\label{Stegenga}
A function $m$ is a multiplier for the Dirichlet space if and only if $m\in H^{\infty}$
and $d\mu_{m}(z)=\left\vert m^{\prime}(z)\right\vert ^{2}dxdy\in CM(\mathcal{D}%
).$
\end{proposition}
This was one of the motivations for Stegenga's study \cite{Ste} of the Carleson measures for the Dirichlet space. Observe that $\int_\DD d\mu_m=\|m\|_{\DDD,*}^2$.

Let us look again at the Hardy case, in the light of Stegenga's Proposition \ref{Stegenga}. Let $\chi_{H^2}$ be the space of the functions $m$ holomorphic in $\DD$ such that the measure
$d\lambda_m(z)=(1-|z|^2)|m^\prime(z)|^2dA(z)$ is a Carleson measure for the Hardy space. The reason for choosing such measure is that $\int_\DD d\lambda_m(z)\approx\|m\|_{H^2}^2$ (if $m(0)=0$), as in the Dirichlet case. Now, it is known that $\chi_{H^2}=BMOA$ is the space of the analytic functions in $BMO$. Proposition \ref{Stegenga} says that multiplier algebra of $\DDD$ is exactly 
$\chi\cap H^\infty$ (here, $\chi$ contains the functions $m$ such that. 
$d\mu_{m}(z)=\left\vert m^{\prime}(z)\right\vert ^{2}dxdy\in CM(\mathcal{D}$). The analogous result for $H^2$ would be that the multiplier space of $H^2$ consists of the functions in $BMOA$ which are essentially bounded. This is true, but not very interesting, since $H^\infty\subseteq BMOA$.
\subsection{\texorpdfstring{The Weakly Factored Space $\DDD\odot\DDD$ and its Dual}{Weakly Factored Space related to the Dirichlet Space and it Dual}}

\subsubsection{\texorpdfstring{Some Facts from $H^2$ Theory.}{Some Facts from Hardy Space Theory}} 

It is well known that some spaces of holomorphic 
functions naturally arise within $H^2$ theory: $H^1$, $H^\infty$, $BMO$. We shortly recall
some of their mutual connections. We have just seen that $H^\infty$ naturally arises as the multiplier algebra of $H^2$:  $Mult(H^2)=H^\infty$. On the other hand, by the inner/outer factorization of
$H^2$ functions it
easily follows that $H^1=H^2\cdot H^2$ is the space of products of $H^2$ functions. 
C. Fefferman's celebrated theorem says that $(H^1)^*=BMO$ is the space of analytic functions
with bounded mean oscillation. Functions in $BMO$ are defined by the well-known, elegant
integral property to which they owe their name, but could be otherwise defined as
the functions $b$ analytic in $\DD$ such that $d\mu_b=(1-|z|^2)|b^\prime(z)|^2dA(z)$ is a 
Carleson measure
for $H^2$:
$$
\int_{\DD}|f|^2d\mu_b\le C(\mu)\|f\|_{H^2}^2.
$$
The spaces just considered are linked with the Hankel forms and Nehari's Theorem. Given
analytic $b$, define the \it Hankel form \rm with symbol $b$ as
$$
T_b(f,g)=\langle b,fg\rangle_{H^2}.
$$
It was shown by Nehari that
$$
\sup_{f,g\in H^2}\frac{|T_b(f,g)|}{\|f\|_{H^2}\|g\|_{H^2}}\approx\|b\|_{(H^1)^*}\approx
\|b\|_{BMO},
$$
the last equality following from Fefferman's Theorem.
\subsubsection{Function Spaces Naturally Related with the Dirichlet Space.}
One might first think that since the Dirichlet space is naturally defined in terms of
hyperbolic geometry the spaces playing the r\^ole of $H^1$, $H^\infty$ and $BMO$ in Dirichlet
theory would be the Bloch space $\BBB$, defined by the (conformally invariant) norm:
$$
\|f\|_{\BBB}=\|\delta f\|_{L^\infty(\DD)}=\sup_{z\in\DD}(1-|z|^2)|f^\prime(z)|
$$
and similarly defined invariant spaces (analytic Besov spaces). It turns out that, from the viewpoint of 
Hilbert space function theory, the relevant spaces are others. An \it a priori \rm reason to guess that Bloch and Besov spaces 
do not play in the Dirichlet theory the r\^ole played by the Hardy spaces $H^p$ ($1\le p\le\infty$) in $H^2$ theory is that inclusions go the wrong way. For instance, $H^\infty\subset H^2$, while $\DDD\subset\BBB$.

\smallskip

Define the weakly factored space $\mathcal{D}\odot\mathcal{D}$
to be the completion of finite sums $h=\sum f_{j}g_{j}$ using the norm
\[
\left\Vert h\right\Vert _{\mathcal{D}\odot\mathcal{D}}=\inf\left\{
\sum\left\Vert f_{j}\right\Vert _{\mathcal{D}}\left\Vert g_{j}\right\Vert
_{\mathcal{D}}:h=\sum f_{j}g_{j}\right\}  .
\]
In particular if $f\in\mathcal{D}$ then $f^{2}\in\mathcal{D}\odot\mathcal{D}$
and
\begin{equation}
\left\Vert f^{2}\right\Vert _{\mathcal{D}\odot\mathcal{D}}\leq\left\Vert
f\right\Vert _{\mathcal{D}}^{2}. \label{square}%
\end{equation}
It is immediate that, in the Hardy case, 
$H^{2}\odot H^{2}=H^2\cdot H^2=H^1$.

We also introduce a variant of $\mathcal{D}\odot\mathcal{D}$. Define the space
$\partial^{-1}\left(  \partial\mathcal{D}\odot\mathcal{D}\right)  $ to be the
completion of the space of functions $h$ such that $h^{\prime}$ can be written
as a finite sum, $h^{\prime}=\sum f_{j}^{\prime}g_{j}$ (and thus
$h=\partial^{-1}\sum\left(  \partial f_{i}\right)  g_{i})$, with the norm
\[
\left\Vert h\right\Vert _{\partial^{-1}\left(  \partial\mathcal{D}%
\odot\mathcal{D}\right)  }=\inf\left\{  \sum\left\Vert f_{j}\right\Vert
_{\mathcal{D}}\left\Vert g_{j}\right\Vert _{\mathcal{D}}:h^{\prime}=\sum
f_{j}^{\prime}g_{j}\right\}  .
\]
We next introduce the space $\mathcal{X}$ which plays a role in the Dirichlet
space theory analogous to the role of $BMO$ in the Hardy space theory. We say
$f\in\mathcal{X}$ if
\[
\left\Vert f\right\Vert _{\mathcal{X}}^{2}=\left\vert f(0)\right\vert
^{2}+\left\Vert \left\vert f^{\prime}\right\vert ^{2}dA\right\Vert
_{CM(\mathcal{D)}}<\infty.
\]
We denote the closure in $\mathcal{X}$ of the space of polynomials by
$\mathcal{X}_{0}.$

Here is a summary of relations between the spaces. The duality pairings are
with respect to the Dirichlet pairing $\left\langle \cdot,\cdot\right\rangle
_{\mathcal{D}}.$

\begin{theorem}
\label{duality} We have

\begin{enumerate}
\item $\mathcal{X}_{0}^{\ast}=\mathcal{D}\odot\mathcal{D}$;

\item $\left(  \mathcal{D}\odot\mathcal{D}\right)  ^{\ast}=\mathcal{X}$;

\item $\mathcal{M}(\mathcal{D})=H^{\infty}\cap\mathcal{X}$;

\item $\mathcal{D}\odot\mathcal{D=}$ $\partial^{-1}\left(  \partial
\mathcal{D}\odot\mathcal{D}\right)  $.
\end{enumerate}

\begin{proof}
[Proof discussion]As we mentioned (3) is proved in \cite{Ste}.

A result essentially equivalent to $\left(  \partial^{-1}\left(
\partial\mathcal{D}\odot\mathcal{D}\right)  \right)  ^{\ast}=\mathcal{X}$ was
proved by Coifman-Muri \cite{CoMu} using real variable techniques and in more
function theoretic contexts by Tolokonnikov \cite{Tolo} and by Rochberg and Wu
 in \cite{RW}. An interesting alternative approach to the result is given by Treil
and Volberg in \cite{TV}.

In \cite{W1} it is shown that $\mathcal{X}_{0}^{\ast}=\partial^{-1}\left(
\partial\mathcal{D}\odot\mathcal{D}\right)  .$ Item (2) is proved in
\cite{ARSW1} and when that is combined with the other results we obtain (1) and (4).
\end{proof}
\end{theorem}

Statement (2) of the theorem is the analog of Nehari's characterization of
bounded Hankel forms on the Hardy space, recast using the identification
$H^{2}\odot H^{2}=H^{1}$ and Fefferman's duality theorem. Item (1) is the
analog of Hartman's characterization of compact Hankel forms. Statement (4) is
similar in spirit to the weak factorization result for Hardy spaces given by
Aleksandrov and Peller in \cite{AP} where they study Foguel-Hankel operators
on the Hardy space.

Given the previous theorem it is easy to check the inclusions%
\begin{equation}
\mathcal{M}(\mathcal{D})\subset\mathcal{X\subset D\subset D}\odot\mathcal{D}
\label{dirinc}%
\end{equation}

In our paper \cite{ARSW2} we discuss more facts about these spaces.
\subsection{The Corona Theorem}

In 1962 Lennart Carleson demonstrated in \cite{Car4} the absence of a corona
in the maximal ideal space of $H^{\infty }$ by
showing that if $\left\{ g_{j}\right\} _{j=1}^{N}$ is a finite set of
functions in $H^{\infty } $ satisfying 
\begin{equation}
\sum_{j=1}^{N}\left\vert g_{j}\left( z\right) \right\vert \geq
\delta>0,\;\;\;\;\;z\in \mathbb{D},  \label{coronadata}
\end{equation}
then there are functions $\left\{ f_{j}\right\} _{j=1}^{N}$ in $H^{\infty
} $ with 
\begin{equation}
\sum_{j=1}^{N}f_{j}\left( z\right) g_{j}\left( z\right) =1,\;\;\;\;\;z\in 
\mathbb{D}.  \label{coronasolutions}
\end{equation}

While not immediately obvious, the result of Carleson is in fact equivalent to the following 
statement about the Hilbert space $H^2$.  If one is given a finite set of functions 
$\{g_j\}_{j=1}^N$ in $H^\infty$ satisfying \eqref{coronadata} and a function 
$h\in H^2$, then there are functions $\{f_j\}_{j=1}^N$ in $H^2$ with
\begin{equation}
\sum_{j=1}^{N}f_{j}\left( z\right) g_{j}\left( z\right) =h(z),\;\;\;\;\;z\in 
\mathbb{D}.  \label{coronasolutions_Hilbert}
\end{equation} 
The key difference between \eqref{coronasolutions} and \eqref{coronasolutions_Hilbert} is that 
one is solving the problem in the Hilbert space setting as opposed to the multiplier algebra, 
which makes the problem somewhat easier.

In this section we discuss the Corona Theorem for the multiplier algebra of the Dirichlet space.  The method of proof will be intimately connected with the resulting statements for $H^\infty$ and $H^2$.  We also will connect this result to a related statement for the Hilbert space $\mathcal{D}$.  The proof of this fact is given by $\overline{\partial}$-methods and the connections between Carleson measures for the space $\mathcal{D}$.  Another proof can be given by simply proving the Hilbert space version directly and then applying the Toeplitz Corona Theorem.  Implicit in both versions are certain solutions to $\overline{\partial}$-problems that arise.

\subsubsection{The $\overline{\partial}$-equation in the Dirichlet Space}
\label{sec:dbar}
As is well-known there is an intimate connection between the Corona Theorem and 
$\overline{\partial}$-problems.  In our context, a $\overline{\partial}$-problem will be to 
solve the following differential equation
\begin{equation}
\label{dbar}
\overline{\partial}b=\mu
\end{equation}
where $\mu$ is a Carleson measure for the space $\mathcal{D}$ and $b$ is some unknown function.  Now solving this problem is an easy application of Cauchy's formula, however we will need to obtain estimates of the solutions.  Tho obtain these estimates, one needs a different solution operator to the $\overline{\partial}$-problem more appropriately suited to our contexts.

In \cite{Xia} Xiao's constructed a non-linear solution operator for \eqref{dbar} 
that is well adapted to solve \eqref{dbar} and obtain estimates.  We note that in the case 
of $H^\infty$ that this result was first obtained by P. Jones, \cite{J}.  
First, note that 
$$
F\left( z\right) =\frac{1}{2\pi i}\iint_{\mathbb{D}}
\frac{d\mu\left( \zeta \right) }{\zeta -z}%
d\zeta \wedge d\overline{\zeta }
$$
satisfies $\overline{\partial }F=\mu$ in the sense of distribution.

The difficulty with this solution kernel is that it does not allow for one to obtain good 
estimates on the solution.  To rectify this, following Jones \cite{J}, we define a new 
non-linear kernel that will overcome this difficulty.  

\begin{theorem}[Jones, \cite{J}]
\label{continuousJones'}
Let $\mu $ be a complex $H^2$ Carleson measure on $\mathbb{D}$. 
Then with $S\left( \mu \right)(z) $ given by
\begin{equation}
\label{Jonessolutionop}
S\left( \mu \right) \left( z\right) =\iint_{\mathbb{D}}K\left(\sigma, z, \zeta\right) 
d\nu \left( \zeta \right)
\end{equation}%
where $\sigma =\frac{\left\vert \mu \right\vert }{\left\Vert \mu \right\Vert
_{CM(H^2)}}$ and%
\begin{equation*}
K\left( \sigma ,z,\zeta \right) \equiv \frac{2i}{\pi }\frac{1-\left\vert
\zeta \right\vert ^{2}}{\left( z-\zeta \right) \left( 1-\overline{\zeta }%
z\right) }\exp \left\{\iint_{\left\vert \omega \right\vert \geq
\left\vert \zeta \right\vert }\left( -\frac{1+\overline{\omega }z}{1-%
\overline{\omega }z}+\frac{1+\overline{\omega }\zeta }{1-\overline{\omega }%
\zeta }\right) d\sigma \left( \omega \right) \right\} ,
\end{equation*}%
we have that:
\begin{enumerate}
\item $S\left( \mu \right) \in L_{loc}^{1}\left( \mathbb{D}\right) $.

\item $\overline{\partial}S\left(\mu\right) =\mu $ in
the sense of distributions.

\item $\iint_{\mathbb{D}}\left\vert K\left( \frac{\left\vert \mu
\right\vert }{\left\Vert \mu \right\Vert _{Car}},x,\zeta \right) \right\vert
d\left\vert \mu \right\vert \left( \zeta \right) \lesssim \left\Vert \mu
\right\Vert _{CM(H^2)}$ for all $x\in \mathbb{T}=\partial \mathbb{D}$,\\
\item[] so $\left\Vert S\left( \mu \right) \right\Vert _{L^{\infty }
\left( \mathbb{T}\right) }\lesssim\left\Vert \mu \right\Vert _{CM(H^2)}$.
\end{enumerate}
\end{theorem}

With this set-up, we now state the following theorem due to Xiao, extending Theorem 
\ref{continuousJones'}, about estimates for $\overline{\partial}$-problems in the 
Dirichlet space.

\begin{theorem}[Xiao, \cite{Xia}]
\label{dbarwithestimates}
If $\left|g(z)\right|^2dA(z)$ is a $\mathcal{D}$-Carleson measure then the function 
$S\left(g(z)dA \right)(z)$ satisfies $\overline{\partial}S( g(z)dA)=g$ and 
$$
\left\| S( g(z)dA)\right\|_{M({\mathcal H}^{1/2}({\mathbb S}))}
\lesssim\| \left|g(z)\right|^2dA(z)\|_{CM(\mathcal{D})}.
$$
Here, $M({\mathcal H}^{1/2}({\mathbb S}))$ is the multiplier algebra of the fractional Sobolev space ${\mathcal H}^{1/2}({\mathbb S})$.
\end{theorem}

\subsubsection{Corona Theorems and Complete Nevanlinna-Pick Kernels}
\label{sec:CoronaNP}

Let $X$ be a Hilbert space of holomorphic functions in an open set $\Omega $
in $\mathbb{C}^{n}$ that is a reproducing kernel Hilbert space with a \emph{%
complete irreducible Nevanlinna-Pick kernel} (see \cite{AgMc2} for the
definition). The following \emph{Toeplitz corona theorem} is due to Ball,
Trent and Vinnikov \cite{BTV} (see also Ambrozie and Timotin \cite{AT}
and Theorem 8.57 in \cite{AgMc2}).

For $f=\left( f_{\alpha }\right) _{\alpha =1}^{N}\in \oplus ^{N}X$ and $h\in
X$, define $\mathbb{M}_{f}h=\left( f_{\alpha }h\right) _{\alpha =1}^{N}$ and%
\begin{equation*}
\left\Vert f\right\Vert _{Mult\left( X,\oplus ^{N}X\right) }=\left\Vert 
\mathbb{M}_{f}\right\Vert _{X\rightarrow \oplus ^{N}X}=\sup_{\left\Vert
h\right\Vert _{X}\leq 1}\left\Vert \mathbb{M}_{f}h\right\Vert _{\oplus
^{N}X}.
\end{equation*}%
Note that $\max_{1\leq \alpha \leq N}\left\Vert \mathcal{M}_{f_{\alpha
}}\right\Vert _{M_{X}}\leq \left\Vert f\right\Vert _{Mult\left( X,\oplus
^{N}X\right) }\leq \sqrt{\sum_{\alpha =1}^{N}\left\Vert \mathcal{M}%
_{f_{\alpha }}\right\Vert _{M_{X}}^{2}}$.

\begin{theorem}[Toeplitz Corona Theorem] 
\label{ToeCorThm}
Let $X$ be a Hilbert function space in an open set $\Omega $ in $\mathbb{C}^{n}$ with an
irreducible complete Nevanlinna-Pick kernel. Let $\delta >0$ and $N\in 
\mathbb{N}$. Then $g_{1},\ldots,g_{N}\in M_{X}$ satisfy the following ``baby
corona property''; for every $h\in X$, there are $f_{1}, \ldots,f_{N}\in X$ such that 
\begin{eqnarray}
\left\Vert f_{1}\right\Vert _{X}^{2}+\cdots+\left\Vert f_{N}\right\Vert
_{X}^{2} &\leq &\frac{1}{\delta }\left\Vert h\right\Vert _{X}^{2},
\label{corcon} \\
g_{1}\left( z\right) f_{1}\left( z\right) +\cdots+g_{N}\left( z\right)
f_{N}\left( z\right) &=&h\left( z\right) ,\ \ \ \ \ z\in \Omega ,  \notag
\end{eqnarray}%
\emph{if and only if} $g_{1},\ldots,g_{N}\in M_{X}$ satisfy the following
``multiplier corona property''; there are $\varphi _{1},\ldots,\varphi _{N}\in
M_{X}$ such that%
\begin{eqnarray}
\left\Vert \varphi \right\Vert _{Mult\left( X,\oplus ^{N}X\right) } &\leq &1,
\label{multcorcon} \\
g_{1}\left( z\right) \varphi _{1}\left( z\right) +\cdots+g_{N}\left( z\right)
\varphi _{N}\left( z\right) &=&\sqrt{\delta },\ \ \ \ \ z\in \Omega .  \notag
\end{eqnarray}
\end{theorem}

The \emph{baby corona theorem} is said to hold for $X$ if whenever $%
g_{1},\cdots,g_{N}\in M_{X}$ satisfy 
\begin{equation}
\left\vert g_{1}\left( z\right) \right\vert ^{2}+\cdots+\left\vert g_{N}\left(
z\right) \right\vert ^{2}\geq c>0,\ \ \ \ \ z\in \Omega ,  \label{lowerphi}
\end{equation}%
then $g_{1},\ldots,g_{N}$ satisfy the baby corona property \eqref{corcon}.

We now state a simple proposition that will be useful in understanding the relationships 
between the Corona problems for $\mathcal{D}$ and $M_{\mathcal{D}}$.

\begin{proposition}
\label{CoronaNASC}
Suppose that $g_1,\ldots, g_N\in M(\mathcal{D})$.  Define the map
$$
M_{(g_1,\ldots, g_n)}(f_1,\ldots, f_n):=\sum_{k=1}^N g_k(z)f_k(z).
$$
Then the following are equivalent
\begin{itemize}
\item[(i)] $M_{(g_1,\ldots, g_n)}:
M(\mathcal{D})\times\cdots\times M(\mathcal{D})\mapsto M(\mathcal{D})$ 
is onto;\label{CoronaMulti}
\item[(ii)] $M_{(g_1,\ldots, g_n)}:\mathcal{D}\times\cdots\times \mathcal{D}
\mapsto \mathcal{D}$ is onto;\label{CoronaHilbert}
\item[(iii)] There exists a $\delta>0$ such that for all $z\in\mathbb{D}$ we have
$$
\sum_{k=1}^N\left|g_k(z)\right|^2\geq\delta>0.
$$
\end{itemize}
\end{proposition}

It is easy to see that both (i) and (ii) each individually imply (iii).  
We will show that condition (iii) implies both (i) and (ii).  Note that by the Toeplitz 
Corona Theorem \ref{ToeCorThm} it would suffice to prove that (iii) implies (ii) since the 
result then lifts to give the statement in (i).  The proof of Proposition \ref{CoronaNASC} 
follows by the lines of Wolff's proof of the Corona Theorem for $H^\infty$, but 
uses the solution operator given by Xiao in Theorem \ref{dbarwithestimates}.

It is important to point out that there are several other proofs of Proposition \ref{CoronaNASC} at 
this point.  The first proof of this fact was given by Tolokonnikov, \cite{Tolo} and was 
essentially obtained via connections with Carleson's Corona Theorem.  Another proof of this 
result, but with the added benefit of being true for an infinite number of generators was 
given by Trent \cite{Trent}.  Trent demonstrated that (iii) implies (ii), and then applied 
the Toeplitz Corona Theorem to deduce that (iii) implies (i).  This proof exploits the fact 
that the kernel for the Dirichlet space is a complete Nevanlinna-Pick kernel.  Finally, there 
is a more recent proof of this fact by Costea, Sawyer and Wick, \cite{CSW}.  The method of 
proof again is to demonstrate the Corona Theorem for $\mathcal{D}$ under the hypothesis (iii).  
The proof in \cite{CSW} is true more generally for the Dirichlet space in any dimension.

\section{Interpolating sequences} Let $\HHH$ be a reproducing kernel Hilbert space (RKHS) 
of functions defined on some space $X$, with kernel functions $\{k_z\}_{z\in X}$. 
Let $\MMM(\HHH)$ be the multiplier space of $\HHH$. A sequence $S\subseteq X$
is an \it interpolating sequence for the multiplier algebra \rm $\MMM(\HHH)$ if the
restriction map
$$
R_S:g\mapsto\{g(s):\ s\in S\}
$$
maps $\MMM(\HHH)$ \it onto \rm $\ell^\infty$. Since $\MMM(\HHH)\subseteq L^\infty(X)$,
the map is automatically bounded. Consider the weight $w:S\to\RR^+$, 
$w(s)=\|k_s\|_\HHH^{-2}$.
We say that the sequence $S$ is \it an interpolating
sequence for the space \rm $\HHH$ if $R_S$ is a bounded map of $\HHH$ \it onto \rm
$\ell^2(S,w)$. In the context of complete Nevanlinna-Pick RKHS these two notions coincide
\cite{MaSu}. Our terminology differs from some sources. In \cite{Bi} Bishop, for instance, calls
\it universally interpolating sequences \rm for $\DDD$ what we call interpolating sequences, and
simply calls \it interpolating sequences what we will call \it onto interpolating sequences\rm.
\begin{theorem}[Marshall and Sundberg, \cite{MaSu}]
\label{toutmeme} 
Let $\HHH$ be a RKHS of functions on some space $X$, with the complete Nevanlinna-Pick property. For a sequence $S$ the following are equivalent:
 \begin{enumerate}
\item $S$ is interpolating for $\MMM(\HHH)$;
\item $S$ is interpolating for $\HHH$;
\item The family of functions $\left\{\frac{k_s}{\|k_s\|_\HHH}\right\}_{s\in S}$
is a Riesz basis for the space $\HHH$:
$$
\left\|\sum_{s\in S}a_s \frac{k_s}{\|k_s\|_\HHH}\right\|_\HHH^2
\approx\sum_{s\in S}|a_s|^2.
$$
 \end{enumerate}
\end{theorem}
Sarason observed that interpolating sequences for the multiplier space $\MMM(\HHH)$ have a 
distinguished r\^ole in the theory of the RKHS space $\HHH$. Let $\varphi$ be a multiplier
of the space $\HHH$ and $S=\{s_j:\ j=1,\dots,n\}$ be a sequence in $X$. Let 
$M_\varphi$ the multiplication operator by $\varphi$ and $M_\varphi^*$ be its adjoint. Then, as we have already seen, 
$\{\overline{\varphi(s_j)}:\ j=1,\dots,n\}$ is a set of eigenvalues for $M_\varphi^*$, 
having the corresponding kernel functions as eigenvectors: 
$M_\varphi^*k_{s_j}=\overline{\varphi(s_j)}k_{s_j}$. 

Finding the multiplier $\varphi$ which interpolates data $\varphi(s_j)=\lambda_j$ corresponds,
then, to extending the diagonal operator $k_{s_j}\mapsto\overline{\lambda_j}k_{s_j}$ 
(which is defined on $\mbox{span}\{k_{s_j}: j=1,\dots,n\}$) to the adjoint of a 
multiplication operator.
We redirect the interested reader to the book \cite{AgMc2}, to the article 
\cite{AgMc1} and to the important manuscript \cite{MaSu} itself for far reaching 
developments of this line of reasoning.

For a given sequence $S$ in $X$, there are two obvious \it necessary conditions \rm for it to be
interpolating for $\HHH$:
\begin{itemize}
\item[\hypertarget{Sep}{(Sep)}] The sequence $S$ is separated: There is a positive $\sigma<1$ such that
for all $s,t\in S$ one has
$$
\left|\left\langle\frac{k_s}{\|k_s\|_\HHH},\frac{k_t}{\|k_t\|_\HHH}\right\rangle\right|\le\sigma.
$$
This condition expresses the fact that there exists a function $f\in\HHH$ such that $f(s)=0$ and $f(t)=1$.
\item[\hypertarget{CM}{(CM)}] The measure $\mu_S=\sum_{s\in S}\|k_s\|_\HHH^{-2}\delta_s$ is a 
Carleson measure for the space $\HHH$:
$$
\int_X|f|^2d\mu_S\le C(\mu)\|f\|_\HHH^2,
$$
which expresses the boundedness of the restriction map $R_S$.
\end{itemize}
Kristian Seip \cite{Seip} conjectures that, for a RKHS with the complete Nevanlinna-Pick property, these two conditions are sufficient for $S$ to be interpolating. Carleson's celebrated Interpolation Theorem \cite{Car2} says that such is the case when $\HHH=H^2$ is the Hardy space. B\"oe proved Seip's conjecture under an additional assumption on the kernel functions (an assumption which, interestingly, is \it not \rm satisfied by the Hardy space itself, but which is satisfied by the Dirichlet space).
  
\subsection{\texorpdfstring{Interpolating Sequences for $\DDD$ and its Multiplier Space}{Interpolating Sequences for the Dirichlet Space and it Multiplier Space}}

The characterization of the interpolating functions was independently solved by Marshall and 
Sundberg \cite{MaSu} and by Bishop \cite{Bi} in 1994.
\begin{theorem}[Bishop \cite{Bi}, Marshall and Sundberg \cite{MaSu}]\label{MaSuBi}
A sequence $S$ in $\DD$ is interpolating for the Dirichlet space $\DDD$ if and only if it 
satisfies \textnormal{\hyperlink{Sep}{(Sep)}} and \textnormal{\hyperlink{CM}{(CM)}}.
\end{theorem}
Actually, Bishop proved that interpolating sequences for $\DDD$ are also interpolating for
$\MMM(\DDD)$, but not the converse. At the present moment, there are four essentially 
different proofs that \hyperlink{Sep}{(Sep)} and \hyperlink{CM}{(CM)} are necessary and sufficient for $\DDD$ interpolation: 
\cite{Bi}, \cite{Bo2}, \cite{Bo1} and \cite{MaSu}.
Interpolating sequences for the Dirichlet space differ in one important aspect from
interpolating sequences for the Hardy space. In the case of $H^2$, in fact, if the restriction
operator is surjective (if, in our terminology, the sequence $S$ is \rm onto interpolating\rm),
then it is automatically bounded. As we will see in the next subsection, there are sequences
$S$ in the unit disc for which the restriction operator is surjective, but not bounded.

It is interesting and useful to restate the separation condition in terms of hyperbolic
distance: \hyperlink{Sep}{(Sep)} \it in the Dirichlet space $\DDD$ holds for the sequence $S$ in $\DD$
if and only if there are positive constants $A,B$ such that, for all $z\ne w\in S$
$$
\max\{d(z),d(w)\}\le Ad(z,w)+B.
$$
\rm This huge separation, which is related to the hyperbolic invariance of the Dirichlet norm,
compensates - in the solution of the interpolating sequences  problem and in other questions -
for the lack of Blaschke products. In fact, it allows much space for crafting
holomorphic functions from smooth ones with little overlap. 

\proof[Proof(s) Discussion] \cite{MaSu}. In their article, Marshall and Sundberg first developed a
general theory concerning interpolating sequences in spaces with the complete
Nevanlinna-Pick property. In particular, they reduced the problem of characterizing
the interpolating sequences for $\DDD$ to that of the interpolating sequences for its multiplier
space. This left them with the (hard) task of showing that, given \hyperlink{Sep}{(Sep)} and \hyperlink{CM}{(CM)}, one could 
interpolate bounded sequences by multiplier functions. In order to do that, they first 
solved the easier (but still difficult) problem of interpolating the data by means of a smooth
function $\varphi:\DD\to\RR$, having properties similar to those of a multiplier in
$\MMM(\DDD)$.  In particular, $\varphi$ is bounded, it has finite Dirichlet norm and, more, 
$|\nabla\varphi|^2dA\in CM(\DDD)$. 

The basic building block for constructing such $\varphi$
are functions $\varphi_z$ attached to points $z$ in $\DD$, which are, substantially,
the best smoothed version (in terms of Dirichlet integral) of the function 
$\chi_{\tilde{S}(z)}$, where 
$$
\tilde{S}(z)=\left\{w\in\DD:\ \left|w-\frac{z}{|z|}\right|\le(1-|z|)^\alpha\right\}
$$ 
($\alpha<1$ suitably chosen) is the ``enlarged Carleson box'' having center in $z$. 
The separation condition \hyperlink{Sep}{(Sep)} ensures that, if $z_1,z_2$ are points of the sequence $S$
and $\mbox{supp}(\varphi_{z_1})\cap \mbox{supp}(\varphi_{z_2})\ne0$, then one of the points
has to be much closer to the boundary than the other. This is one of the two main tools (the other being the Carleson measure condition, which further separates the points of the sequence) in the various estimates for linear combinations of basic building functions. These basic building blocks and their holomorphic modifications are the main tool in the proofs of the interpolating theorems in \cite{ARS5} and \cite{Bo1}.  The rest of the proof consists in showing that one can correct the function $\varphi$, making it harmonic, and from this, one easily proceeds to the holomorphic case.

\smallskip

Bishop, instead, uses as building blocks conformal maps, (see \cite{Bi}, p.27). In his article, he observes that the construction of the interpolating functions for $\DDD$ does not require the full use of the assumption \hyperlink{CM}{(CM)}.  This is contrary to the Hardy case, where there are sequences $S$ for which the restriction operator is \it  onto and unbounded. \rm We will return on this in the next section. 

B\"oe's short proof in \cite{Bo2} is less constructive, and it relies on Hilbert space 
arguments. However, in his paper \cite{Bo1}, dealing with the more general case of the
analytic Besov spaces, B\"oe has an explicit construction of the interpolating sequences
for $\DDD$ and $\MMM(\DDD)$. 
He makes use of holomorphic modifications of the functions $\varphi_z$ in \cite{MaSu}, 
which are the starting point for a hard and clever recursion scheme.

It is clear from the construction in \cite{Bo2} that, under the assumptions \hyperlink{Sep}{(Sep)} and \hyperlink{CM}{(CM)}, one
has \it linear interpolation of data\rm, both in $\DDD$ and $\MMM(\DDD)$: there exist
linear operators $L_S:\ell^\infty(S)\to\MMM(\DDD)$ and 
$\Lambda_S:\ell^2(S,w)\to\DDD$ such that $L_S\{a_s:\ s\in S\}$ solves the interpolating problem
in $\MMM(\DDD)$ with data $\{a_s:\ s\in S\}\in\ell^\infty(S)$ and 
 $\Lambda_S\{b_s:\ s\in S\}$ solves the interpolating problem
in $\DDD$ with data $\{b_s:\ s\in S\}\in \ell^2(S,w)$. 
\endproof
\subsection{Weak Interpolation and ``Onto'' Interpolation} A sequence $S$ in $\DD$ is 
\it onto interpolating \rm if the restriction operator $R_S$ maps $\DDD$ onto $\ell^2(S,w)$.
We do not ask the operator $R_S$ to be bounded (hence, to be defined on all of $\DDD$). 
It follows from the Closed Graph Theorem that, if $S$ is onto interpolating, then it is 
interpolating \it with norm control.  Namely, \rm there is a constant $C>0$ such that for 
$\{a_s:\ s\in S\}\in\ell^2(S,w)$ there is $f\in\DDD$ such that $f(s)=a_s$ and 
$\|f\|_\DDD\le C\|\{a_s:\ s\in S\}\|_{\ell^2(S,w)}$. Furthermore, interpolation can be realized
linearly.

A sequence $S$ in $\DD$ is \it weakly interpolating \rm if there is $C>0$ such that, for all 
$s_0\in S$ there is $f_{s_0}\in\DDD$ with  $f_{s_0}(s)=\delta_{s_0}(s)$ for $s\in S$
($\delta_{s_0}$ is the Kroenecker function) and norm control
$\|f_{s_0}\|_\DDD^2\le C\|\delta_{s_0}\|_{\ell^2(S,w)}^2\approx d(s_0)^{-1}$. 
Clearly, weakly interpolating is weaker than onto interpolating.
\begin{remark}
\label{addsomepoints}

\begin{itemize}
\item[]
 \item[(a)] Weak interpolation ({\rm a fortiori}, onto interpolation) implies the separation
condition \textnormal{\hyperlink{Sep}{(Sep)}};
\item[(b)] By adding a finite number of points to an onto interpolating sequence, we obtain
another onto interpolating sequence.
\end{itemize}
\end{remark}

A geometric characterization of the onto interpolating sequences is still lacking. However,
the following facts are known.
\begin{theorem}[Bishop \cite{Bi}]\label{weaklybishop}
The sequence $S$ is onto interpolating if and only if it is weakly interpolating, and this is
in turn equivalent to having weak interpolation with functions which satisfy the further
condition that $\|f_s\|_{L^\infty(\DDD)}\le C$ for some constant $C$.
\end{theorem}
The proof of Bishop's Theorem involves the clever use of a variety of sophisticated tools.
It would be interesting having a different proof (one which worked for the analytic Besov
spaces, for instance). Unfortunately, establishing whether a given sequence $S$ is weakly
interpolating is not much easier than establishing if it is onto interpolating.

Both Bishop \cite{Bi} and B\"oe realized that a sequence $S$ is onto interpolating if the 
associated measure $\mu_S$ satisfies the \it simple condition \rm \eqref{simplecondition}
instead of the stronger Carleson measure condition \hyperlink{CM}{(CM)}. The simple condition implies, in
particular, that the measure $\mu_S$ is finite and Bishop asked whether there are 
onto interpolating sequences with infinite $\mu_S$. The answer is affirmative:
\begin{theorem}[Arcozzi, Rochberg, and Sawyer, \cite{ARS5}]\label{ontoarsone}
There exist sequences $S$ in $\DD$ with $\mu_S(\DD)=+\infty$, 
which are onto interpolating for $\DDD$.
\end{theorem}
The proof of Theorem \ref{ontoarsone} relies on a modification of B\"oe's recursive scheme,
using B\"oe's functions.   In \cite{ARS5} there is another partial result, which extends the 
theorems of Bishop and B\"oe. In order to state it, we have to go back to the tree language. 
Let $\TTT$ be the dyadic tree associated with the disc $\DD$. By Remark \ref{addsomepoints},
we can assume that each box $\alpha$ in $\TTT$ contains at most one point from the
candidate interpolating sequence $S$ in $\DD$. We can therefore identify points in $S$
with distinguished boxes in $\TTT$.
We say that $S$ in $\DD$ satisfies the \it weak simple condition \rm if
for all $\alpha$ in $\TTT$,
\begin{equation}
\label{weaksimple}
\sum_{\substack{\beta\in S,\text{ }\beta\geq\alpha\\\mu_S\left(  \gamma\right)
=0\text{ for }\alpha<\gamma<\beta}}\mu_S\left(  \beta\right)  \leq Cd\left(
\alpha\right)  ^{-1}. 
\end{equation}
\begin{theorem}[Arcozzi, Rochberg and Sawyer \cite{ARS5}]\label{intweaksimple}
Let $S$ be a sequence in $\DD$ and suppose that $\mu_S(\DD)<\infty$. If $\mu_S$ satisfies
\textnormal{\hyperlink{Sep}{(Sep)}} and the weak simple condition \eqref{weaksimple}, then $S$ is onto interpolating for $\DDD$.
\end{theorem}
We observe that the weak simple condition can not be necessary for onto interpolation. 
In fact, it is easy to
produce examples of sequences $S$ satisfying \eqref{weaksimple}, having subsequences
$S^\prime$ (which are then onto interpolating for $\DDD$) for which 
\eqref{weaksimple} is not satisfied. Such examples, however, have $\mu_S(\DD)=+\infty$. We 
do not know whether, under the assumptions $\mu_S(\DD)=+\infty$ and \hyperlink{Sep}{(Sep)}, the weak simple 
condition is necessary for onto interpolation.

In terms of partial results about onto interpolation, let us mention a necessary condition of 
capacitary type. If $S$ an onto interpolating sequence for $\DDD$ in $\DD$, which we might identify with a 
subsequence of the tree $\TTT$, then the \it discrete capacitary condition holds\rm:  to each $s_0$ in $S$, there corresponds a positive function $\varphi_{s_0}$ on $\TTT$ such that $\|\varphi_{s_0}\|_{\ell^2}^2\le Cd(s_0)^{-1}$ and $\varphi_{s_0}(s)=\delta_{s_0}(s)$ whenever $s\in S$.

\smallskip

A proof of this fact easily follows from Proposition \ref{restriction}.
The discrete capacitary condition can be stated in terms of discrete condenser capacities:
$$
\capacity_\TTT(S\setminus\{s_0\},\{s_0\})\le Cd(s_0)^{-1}.
$$
A reasonable conjecture is that the discrete capacitary condition, plus the separation 
condition \hyperlink{Sep}{(Sep)}, are necessary and sufficient in order for $S$ to be onto interpolating.
Other material on the problem of the onto interpolating sequences is in \cite{ARS5}.

\subsection{Zero sets}\label{zeroincondotta} We briefly mention, because related to the interpolating sequences, the zero sets for Dirichlet functions. A sequence of points $Z=\{z_n:\ n\ge0\}$ in $\DD$ is a \it zero set \rm for $\DDD$ if there is a nonvanishing function $f$ in $\DDD$ such that $f(z_n)=0$. By conformal invariance, we might also require $f(0)=1$. In \cite{SS2} Shapiro and Shields, improving on a Theorem of Carleson \cite{Car1}, proved that if
\begin{equation}\label{zerosuzero}
\sum_n\frac{1}{\log\frac{1}{1-|z_n|}}<\infty,
\end{equation}
then $Z$ is a zero set for the Dirichlet spaces. This condition, shown in \cite{NRS}, is sharp among conditions which only depend on the distance from the origin; but it does not characterize the zero sets. We direct the interested reader to \cite{Ro} for more information on zero sets.
\section{Some open problems.}\label{aperto} We conclude this survey with some open problems strictly related to the topics we have discussed.

Since the Dirichlet space is conformally invariant, it would be interesting to have conformally invariant counterparts of definitions and theorems concerning the Dirichlet space, in which the origin plays a privileged r\^ole.  A natural conformally invariant \it definition \rm of the Carleson measure norm for a measure $\mu$ on $\overline{\DD}$ is
$$
[\mu]_{CM_{inv}(\DDD)}:=\sup\frac{\int_{\overline{\DD}}|f-\mu(f)|^2d\mu }{\|f\|_{\DDD,*}^2},
$$
where $\mu(f)=\int_{\overline{\DD}}fd\mu$. Given a M\"obius map of the disc $\varphi$, let $\varphi_*\mu$, defined as $\varphi_*\mu(E)=\mu(\varphi^{-1}(E))$, be the push forward measure. It is easily verified that 
$[\varphi_*\mu]_{CM_{inv}(\DDD)}=[\mu]_{CM_{inv}(\DDD)}$. 
\begin{problem} Give a quantitative, geometric characterization of $[\mu]_{CM_{inv}(\DDD)}$.
\end{problem}
Let $[\mu]_{CM(\DDD)}$ be the best constant in the Carleson imbedding inequality \eqref{carleson}.
It is proved in \cite{ARS1} that (if $\mu(\partial \DD)=0$) then $[\mu]_{CM_{inv}(\DDD)}$ is finite if and only if $[\mu]_{CM(\DDD)}$ is finite. The proof there, however, is by contradiction and does not seem to give quantitative information. Also note that the case of the point mass $\mu=\delta_0$ shows that $[\mu]_{CM(\DDD)}$ and $[\mu]_{CM_{inv}(\DDD)}$ are not equivalent in general.
It would also be interesting to have a conformally invariant definition and geometric characterization of the interpolating sequences for the Dirichlet space. 

The circle of ideas revolving around Nehari's Theorem and duality  is now established for the Hardy space $H^2$ and for the Dirichlet space $\DDD$. Similar, deep results have been obtained for spaces which are not holomorphic spaces on the unit disc. 
\begin{problem}
Does the same theory hold for the weighted Dirichlet spaces sitting between Hardy and Dirichlet?
\end{problem}
The weighted Dirichlet spaces we are referring to are those semi-normed by
$$
\|f\|^2_{\DDD_a,*}=\int_\DD|f^\prime|^2d(1-|z|^2)^aA(z),
$$
where $0< a< 1$: $\DDD_0=\DDD$ and $\DDD_1=H^2$. The techniques in \cite{ARSW1} can not be directly applied to the weighted case.

About interpolating sequences, the following problem is still open:
\begin{problem}
Give a geometric characterization of the onto-interpolating sequences for the Dirichlet space.
\end{problem}
Some results in \cite{ARS5} seem to imply that, in order to solve this problem, one has to depart from B\"oe's constructive techniques \cite{Bo1}.
Onto interpolation is related to the following old problem.
\begin{problem}
Characterize the zero sets for $\DDD$.
\end{problem}

The interpretation of the Dirichlet norm in terms of area of the image provides a natural, conformal invariant definition of the Dirichlet norm on any planar domain.
\begin{problem}
Develop a theory of Dirichlet spaces on planar domains. 
\end{problem}
Of special interest, in view of potential applications to condenser capacities, would be a theory of  Dirichlet spaces on annuli.


\begin{bibdiv}
\begin{biblist}
\normalsize
\bib{Adams}{article}{
   author={Adams, David R.},
   title={On the existence of capacitary strong type estimates in $R^{n}$},
   journal={Ark. Mat.},
   volume={14},
   date={1976},
   number={1},
   pages={125--140}
}

\bib{AgMc2}{book}{
   author={Agler, Jim},
   author={McCarthy, John E.},
   title={Pick interpolation and Hilbert function spaces},
   series={Graduate Studies in Mathematics},
   volume={44},
   publisher={American Mathematical Society},
   place={Providence, RI},
   date={2002},
   pages={xx+308}
}
		
\bib{AgMc1}{article}{
   author={Agler, Jim},
   author={McCarthy, John E.},
   title={Complete Nevanlinna-Pick kernels},
   journal={J. Funct. Anal.},
   volume={175},
   date={2000},
   number={1},
   pages={111--124}
}

\bib{AP}{article}{
   author={Aleksandrov, A. B.},
   author={Peller, V. V.},
   title={Hankel operators and similarity to a contraction},
   journal={Internat. Math. Res. Notices},
   date={1996},
   number={6},
   pages={263--275}
}

\bib{AT}{article}{
   author={Ambrozie, C{\u{a}}lin-Grigore},
   author={Timotin, Dan},
   title={On an intertwining lifting theorem for certain reproducing kernel
   Hilbert spaces},
   journal={Integral Equations Operator Theory},
   volume={42},
   date={2002},
   number={4},
   pages={373--384}
}
		
\bib{AF}{article}{
   author={Arazy, J.},
   author={Fisher, S. D.},
   title={The uniqueness of the Dirichlet space among M\"obius-invariant
   Hilbert spaces},
   journal={Illinois J. Math.},
   volume={29},
   date={1985},
   number={3},
   pages={449--462}
}

\bib{AFP}{article}{
   author={Arazy, J.},
   author={Fisher, S. D.},
   author={Peetre, J.},
   title={M\"obius invariant function spaces},
   journal={J. Reine Angew. Math.},
   volume={363},
   date={1985},
   pages={110--145}
}

\bib{A}{article}{
   author={Arcozzi, Nicola},
   title={Carleson measures for analytic Besov spaces: the upper triangle
   case},
   journal={JIPAM. J. Inequal. Pure Appl. Math.},
   volume={6},
   date={2005},
   number={1},
   pages={Article 13, 15 pp. (electronic)}
}

\bib{AR}{article}{
   author={Arcozzi, Nicola},
   author={Rochberg, Richard},
   title={Topics in dyadic Dirichlet spaces},
   journal={New York J. Math.},
   volume={10},
   date={2004},
   pages={45--67 (electronic)}
}

\bib{ARS2}{article}{
   author={Arcozzi, N.},
   author={Rochberg, R.},
   author={Sawyer, E.},
   title={Carleson measures for the Drury-Arveson Hardy space and other
   Besov-Sobolev spaces on complex balls},
   journal={Adv. Math.},
   volume={218},
   date={2008},
   number={4},
   pages={1107--1180}
}

\bib{ARS4}{article}{
   author={Arcozzi, N.},
   author={Rochberg, R.},
   author={Sawyer, E.},
   title={Capacity, Carleson measures, boundary convergence, and exceptional
   sets},
   conference={
      title={Perspectives in partial differential equations, harmonic
      analysis and applications},
   },
   book={
      series={Proc. Sympos. Pure Math.},
      volume={79},
      publisher={Amer. Math. Soc.},
      place={Providence, RI},
   },
   date={2008},
   pages={1--20}
}

\bib{ARS5}{article}{
   author={Arcozzi, N.},
   author={Rochberg, R.},
   author={Sawyer, E.},
   title={Onto interpolating sequences for the Dirichlet space},
   pages={preprint},
   eprint={http://amsacta.cib.unibo.it/2480/}
}

\bib{ARS3}{article}{
   author={Arcozzi, Nicola},
   author={Rochberg, Richard},
   author={Sawyer, Eric},
   title={The characterization of the Carleson measures for analytic Besov
   spaces: a simple proof},
   conference={
      title={Complex and harmonic analysis},
   },
   book={
      publisher={DEStech Publ., Inc., Lancaster, PA},
   },
   date={2007},
   pages={167--177}
}

\bib{ARS6}{article}{
   author={Arcozzi, N.},
   author={Rochberg, R.},
   author={Sawyer, E.},
   title={Carleson measures and interpolating sequences for Besov spaces on
   complex balls},
   journal={Mem. Amer. Math. Soc.},
   volume={182},
   date={2006},
   number={859},
   pages={vi+163}
}

\bib{ARS1}{article}{
   author={Arcozzi, Nicola},
   author={Rochberg, Richard},
   author={Sawyer, Eric},
   title={Carleson measures for analytic Besov spaces},
   journal={Rev. Mat. Iberoamericana},
   volume={18},
   date={2002},
   number={2},
   pages={443--510}
}

\bib{ARSW1}{article}{
   author={Arcozzi, Nicola},
   author={Rochberg, Richard},
   author={Sawyer, Eric},
   author={Wick, Brett D.},
   title={Bilinear forms on the Dirichlet space},
   journal={Anal. \& PDE},
   volume={3},
   date={2010},
   number={1},
   pages={21--47}
}

\bib{ARSW2}{article}{
   author={Arcozzi, Nicola},
   author={Rochberg, Richard},
   author={Sawyer, Eric},
   author={Wick, Brett D.},
   title={Function spaces related to the Dirichlet space},
   journal={J. London Math. Soc.},
   eprint={http://arxiv.org/abs/0812.3422}
}

\bib{ARSW3}{article}{
   author={Arcozzi, Nicola},
   author={Rochberg, Richard},
   author={Sawyer, Eric},
   author={Wick, Brett D.},
   title={Potential Theory on Trees, Graphs and Ahlfors regular metric spaces},
   pages={preprint}
}

\bib{BTV}{article}{
   author={Ball, Joseph A.},
   author={Trent, Tavan T.},
   author={Vinnikov, Victor},
   title={Interpolation and commutant lifting for multipliers on reproducing
   kernel Hilbert spaces},
   conference={
      title={Operator theory and analysis},
      address={Amsterdam},
      date={1997},
   },
   book={
      series={Oper. Theory Adv. Appl.},
      volume={122},
      publisher={Birkh\"auser},
      place={Basel},
   },
   date={2001},
   pages={89--138}
}

\bib{Beu}{article}{
   author={Beurling, Arne},
   title={Ensembles exceptionnels},
   language={French},
   journal={Acta Math.},
   volume={72},
   date={1940},
   pages={1--13}
}

\bib{Bi}{article}{
   author={Bishop, C. J.},
   title={Interpolating sequences for the Dirichlet space and its multipliers},
   pages={preprint},
   eprint={http://www.math.sunysb.edu/~bishop/papers/mult.pdf}
}

\bib{BePe}{article}{
   author={Benjamini, Itai},
   author={Peres, Yuval},
   title={A correlation inequality for tree-indexed Markov chains},
   conference={
      title={Seminar on Stochastic Processes, 1991},
      address={Los Angeles, CA},
      date={1991},
   },
   book={
      series={Progr. Probab.},
      volume={29},
      publisher={Birkh\"auser Boston},
      place={Boston, MA},
   },
   date={1992},
   pages={7--14}
}

\bib{BlPa}{article}{
   author={Blasi, Daniel},
   author={Pau, Jordi},
   title={A characterization of Besov-type spaces and applications to
   Hankel-type operators},
   journal={Michigan Math. J.},
   volume={56},
   date={2008},
   number={2},
   pages={401--417}
}



\bib{Bo2}{article}{
   author={B{\"o}e, Bjarte},
   title={A norm on the holomorphic Besov space},
   journal={Proc. Amer. Math. Soc.},
   volume={131},
   date={2003},
   number={1},
   pages={235--241 (electronic)}
}
		
\bib{Bo1}{article}{
   author={B{\"o}e, Bjarte},
   title={Interpolating sequences for Besov spaces},
   journal={J. Funct. Anal.},
   volume={192},
   date={2002},
   number={2},
   pages={319--341}
}

\bib{Car4}{article}{
   author={Carleson, Lennart},
   title={Interpolations by bounded analytic functions and the corona
   problem},
   journal={Ann. of Math. (2)},
   volume={76},
   date={1962},
   pages={547--559}
}
		
\bib{Car3}{article}{
   author={Carleson, Lennart},
   title={A representation formula for the Dirichlet integral},
   journal={Math. Z.},
   volume={73},
   date={1960},
   pages={190--196}
}
		
\bib{Car2}{article}{
   author={Carleson, Lennart},
   title={An interpolation problem for bounded analytic functions},
   journal={Amer. J. Math.},
   volume={80},
   date={1958},
   pages={921--930}
}
		
\bib{Car1}{article}{
   author={Carleson, Lennart},
   title={On the zeros of functions with bounded Dirichlet integrals},
   journal={Math. Z.},
   volume={56},
   date={1952},
   pages={289--295}
}

\bib{CoMu}{article}{
   author={Coifman, Ronald R.},
   author={Murai, Takafumi},
   title={Commutators on the potential-theoretic energy spaces},
   journal={Tohoku Math. J. (2)},
   volume={40},
   date={1988},
   number={3},
   pages={397--407}
}
		
\bib{CSW}{article}{
   author={Costea, \c{S}erban},
   author={Sawyer, Eric T.},
   author={Wick, Brett D.},
   title={The Corona Theorem for the Drury-Arveson Hardy space and other holomorphic Besov-Sobolev spaces on the unit ball in $\mathbb{C}^{n}$},
   journal={Anal. \& PDE},
   pages={to appear},
   eprint={http://arxiv.org/abs/0811.0627}
}

\bib{HW}{article}{
   author={Hedberg, L. I.},
   author={Wolff, Th. H.},
   title={Thin sets in nonlinear potential theory},
   journal={Ann. Inst. Fourier (Grenoble)},
   volume={33},
   date={1983},
   number={4},
   pages={161--187}
}

\bib{J}{article}{
   author={Jones, Peter W.},
   title={$L^{\infty }$ estimates for the $\bar \partial $ problem in a
   half-plane},
   journal={Acta Math.},
   volume={150},
   date={1983},
   number={1-2},
   pages={137--152}
}

\bib{KaVe}{article}{
   author={Kalton, N. J.},
   author={Verbitsky, I. E.},
   title={Nonlinear equations and weighted norm inequalities},
   journal={Trans. Amer. Math. Soc.},
   volume={351},
   date={1999},
   number={9},
   pages={3441--3497}
}
		
\bib{KS1}{article}{
   author={Kerman, Ron},
   author={Sawyer, Eric},
   title={Carleson measures and multipliers of Dirichlet-type spaces},
   journal={Trans. Amer. Math. Soc.},
   volume={309},
   date={1988},
   number={1},
   pages={87--98}
}

\bib{KS2}{article}{
   author={Kerman, R.},
   author={Sawyer, E.},
   title={The trace inequality and eigenvalue estimates for Schr\"odinger
   operators},
   language={English, with French summary},
   journal={Ann. Inst. Fourier (Grenoble)},
   volume={36},
   date={1986},
   number={4},
   pages={207--228}
}

\bib{MaSu}{article}{
   author={Marshall, D.},
   author={Sundberg, C.},
   title={Interpolating sequences for the multipliers of the Dirichlet space},
   pages={preprint},
   eprint={http://www.math.washington.edu/~marshall/preprints/interp.pdf}
}

\bib{Maz}{article}{
   author={Maz{\cprime}ja, V. G.},
   title={Certain integral inequalities for functions of several variables},
   language={Russian},
   conference={
      title={Problems of mathematical analysis, No. 3: Integral and
      differential operators, Differential equations (Russian)},
   },
   book={
      publisher={Izdat. Leningrad. Univ., Leningrad},
   },
   date={1972},
   pages={33--68}
}

\bib{MW}{article}{
   author={Muckenhoupt, Benjamin},
   author={Wheeden, Richard},
   title={Weighted norm inequalities for fractional integrals},
   journal={Trans. Amer. Math. Soc.},
   volume={192},
   date={1974},
   pages={261--274}
}

\bib{NaSul}{article}{
   author={Nag, Subhashis},
   author={Sullivan, Dennis},
   title={Teichm\"uller theory and the universal period mapping via quantum
   calculus and the $H^{1/2}$ space on the circle},
   journal={Osaka J. Math.},
   volume={32},
   date={1995},
   number={1},
   pages={1--34}
}

\bib{NRS}{article}{
   author={Nagel, Alexander},
   author={Rudin, Walter},
   author={Shapiro, Joel H.},
   title={Tangential boundary behavior of functions in Dirichlet-type
   spaces},
   journal={Ann. of Math. (2)},
   volume={116},
   date={1982},
   number={2},
   pages={331--360}
}

\bib{RiSu}{article}{
   author={Richter, Stefan},
   author={Sundberg, Carl},
   title={A formula for the local Dirichlet integral},
   journal={Michigan Math. J.},
   volume={38},
   date={1991},
   number={3},
   pages={355--379}
}

\bib{RW}{article}{
   author={Rochberg, Richard},
   author={Wu, Zhi Jian},
   title={A new characterization of Dirichlet type spaces and applications},
   journal={Illinois J. Math.},
   volume={37},
   date={1993},
   number={1},
   pages={101--122}
}

\bib{Ro}{article}{
   author={Ross, William T.},
   title={The classical Dirichlet space},
   conference={
      title={Recent advances in operator-related function theory},
   },
   book={
      series={Contemp. Math.},
      volume={393},
      publisher={Amer. Math. Soc.},
      place={Providence, RI},
   },
   date={2006},
   pages={171--197}
}

\bib{Sar}{article}{
   author={Sarason, Donald},
   title={Holomorphic spaces: a brief and selective survey},
   conference={
      title={Holomorphic spaces},
      address={Berkeley, CA},
      date={1995},
   },
   book={
      series={Math. Sci. Res. Inst. Publ.},
      volume={33},
      publisher={Cambridge Univ. Press},
      place={Cambridge},
   },
   date={1998},
   pages={1--34}
}

\bib{Saw}{article}{
   author={Sawyer, Stanley A.},
   title={Martin boundaries and random walks},
   conference={
      title={Harmonic functions on trees and buildings},
      address={New York},
      date={1995},
   },
   book={
      series={Contemp. Math.},
      volume={206},
      publisher={Amer. Math. Soc.},
      place={Providence, RI},
   },
   date={1997},
   pages={17--44}
}

\bib{Seip}{book}{
   author={Seip, Kristian},
   title={Interpolation and sampling in spaces of analytic functions},
   series={University Lecture Series},
   volume={33},
   publisher={American Mathematical Society},
   place={Providence, RI},
   date={2004},
   pages={xii+139}
}

\bib{SS2}{article}{
   author={Shapiro, H. S.},
   author={Shields, A. L.},
   title={On the zeros of functions with finite Dirichlet integral and some
   related function spaces},
   journal={Math. Z.},
   volume={80},
   date={1962},
   pages={217--229}
}
		

\bib{SS1}{article}{
   author={Shapiro, H. S.},
   author={Shields, A. L.},
   title={On some interpolation problems for analytic functions},
   journal={Amer. J. Math.},
   volume={83},
   date={1961},
   pages={513--532}
}
		
\bib{Ste}{article}{
   author={Stegenga, David A.},
   title={Multipliers of the Dirichlet space},
   journal={Illinois J. Math.},
   volume={24},
   date={1980},
   number={1},
   pages={113--139}
}

\bib{Stein}{book}{
   author={Stein, Elias M.},
   title={Singular integrals and differentiability properties of functions},
   series={Princeton Mathematical Series, No. 30},
   publisher={Princeton University Press},
   place={Princeton, N.J.},
   date={1970},
   pages={xiv+290}
}

\bib{Tch}{article}{
   author={Tchoundja, Edgar},
   title={Carleson measures for the generalized Bergman spaces via a
   $T(1)$-type theorem},
   journal={Ark. Mat.},
   volume={46},
   date={2008},
   number={2},
   pages={377--406}
}

\bib{Tolo}{article}{
   author={Tolokonnikov, V. A.},
   title={Carleson's Blaschke products and Douglas algebras},
   language={Russian},
   journal={Algebra i Analiz},
   volume={3},
   date={1991},
   number={4},
   pages={186--197},
   issn={0234-0852},
   translation={
      journal={St. Petersburg Math. J.},
      volume={3},
      date={1992},
      number={4},
      pages={881--892},
      issn={1061-0022},
   }
}

\bib{TV}{article}{
   author={Treil, Sergei},
   author={Volberg, Alexander},
   title={A fixed point approach to Nehari's problem and its applications},
   conference={
      title={Toeplitz operators and related topics},
      address={Santa Cruz, CA},
      date={1992},
   },
   book={
      series={Oper. Theory Adv. Appl.},
      volume={71},
      publisher={Birkh\"auser},
      place={Basel},
   },
   date={1994},
   pages={165--186}
}

\bib{Trent}{article}{
   author={Trent, Tavan T.},
   title={A corona theorem for multipliers on Dirichlet space},
   journal={Integral Equations Operator Theory},
   volume={49},
   date={2004},
   number={1},
   pages={123--139}
}

\bib{Tw}{article}{
   author={Twomey, J. B.},
   title={Tangential boundary behaviour of harmonic and holomorphic
   functions},
   journal={J. London Math. Soc. (2)},
   volume={65},
   date={2002},
   number={1},
   pages={68--84}
}

\bib{V}{article}{
   author={Verbitski{\u\i}, I. {\`E}.},
   title={Multipliers in spaces with ``fractional'' norms, and inner
   functions},
   language={Russian},
   journal={Sibirsk. Mat. Zh.},
   volume={26},
   date={1985},
   number={2},
   pages={51--72, 221}
}

\bib{VoWi}{article}{
   author={Volberg, Alexander},
   author={Wick, Brett D.},
   title={Bergman-type Singular Operators and the Characterization of Carleson Measures for Besov--Sobolev Spaces on the Complex Ball},
   journal={Amer. J. Math.},
   pages={to appear},
   eprint={http://arxiv.org/abs/0910.1142}
}

\bib{Wu}{article}{
   author={Wu, Zhijian},
   title={Function theory and operator theory on the Dirichlet space},
   conference={
      title={Holomorphic spaces},
      address={Berkeley, CA},
      date={1995},
   },
   book={
      series={Math. Sci. Res. Inst. Publ.},
      volume={33},
      publisher={Cambridge Univ. Press},
      place={Cambridge},
   },
   date={1998},
   pages={179--199}
}
 
\bib{W1}{article}{
   author={Wu, Zhijian},
   title={The predual and second predual of $W_\alpha$},
   journal={J. Funct. Anal.},
   volume={116},
   date={1993},
   number={2},
   pages={314--334}
}

\bib{Xia}{article}{
   author={Xiao, Jie},
   title={The $\overline\partial$-problem for multipliers of the Sobolev
   space},
   journal={Manuscripta Math.},
   volume={97},
   date={1998},
   number={2},
   pages={217--232}
}

\bib{Zhu}{book}{
   author={Zhu, Kehe},
   title={Operator theory in function spaces},
   series={Mathematical Surveys and Monographs},
   volume={138},
   edition={2},
   publisher={American Mathematical Society},
   place={Providence, RI},
   date={2007},
   pages={xvi+348}
}

\end{biblist}
\end{bibdiv}
\end{document}